\documentclass[11pt,a4paper]{article}
\usepackage[utf8]{inputenc}
\usepackage{microtype}
\usepackage{amsmath,amsthm,amssymb}
\usepackage{mathrsfs} 
\usepackage[T1]{fontenc}
\usepackage{slashed}
\usepackage{times}
\usepackage[british]{babel}
\usepackage{aliascnt}
\usepackage[colorlinks,pagebackref,hypertexnames=false]{hyperref} 
\usepackage{url}
\usepackage[numeric,backrefs]{amsrefs}
\usepackage{nth}
\usepackage{float}
\usepackage[all]{xy}
\usepackage{tikz-cd} 
\usepackage{bookmark}
\usepackage[margin=2cm]{geometry}
\usepackage{enumerate}
\usepackage{soul}

\newcommand{\rG}{{\rm G}}

\newcommand{\R}{\mathbb{R}}
\newcommand{\C}{\mathbb{C}}
\renewcommand{\H}{\bH}

\newcommand{\su}{\mathfrak{su}}
\renewcommand{\so}{\mathfrak{so}}

\newcommand{\SO}{{\rm SO}}

\newcommand{\SU}{{\rm SU}}

\newcommand{\SL}{\mathrm{SL}}
\newcommand{\U}{{\rm U}}

\renewcommand{\epsilon}{\varepsilon}

\newcommand{\diag}{\mathrm{diag}}

\newcommand{\vol}{\mathrm{vol}}

\newcommand{\qandq}{\quad\text{and}\quad}

\def\<{\mathopen{}\left<}
\def\>{\right>\mathclose{}}
\def\({\mathopen{}\left(}
\def\){\right)\mathclose{}}

\usepackage{multicol, color}

\definecolor{gold}{rgb}{0.85,.66,0}
\definecolor{cherry}{rgb}{0.9,.1,.2}
\definecolor{burgundy}{rgb}{0.8,.2,.2}
\definecolor{orangered}{rgb}{0.85,.3,0}
\definecolor{orange}{rgb}{0.85,.4,0}
\definecolor{olive}{rgb}{.45,.4,0}
\definecolor{lime}{rgb}{.6,.9,0}
\definecolor{green}{rgb}{.2,.7,0}
\definecolor{grey}{rgb}{.4,.4,.2}
\definecolor{brown}{rgb}{.4,.3,.1}

\def\rmG{\mathrm{G}}
\def\al{\alpha}

\def\d{\text{d}}
\def\w{\wedge}
\def\R{\mathbb{R}}
\def\C{\mathbb{C}}

\def\Lm{\Lambda}
\def\om{\omega}
\def\Om{\Omega}
\def\vp{\varphi}
\def\ip{\raise1pt\hbox{\large$\lrcorner\ \!$}} 

\def\G{\mathrm{G}}
\newcommand{\nc}{\newcommand}
\nc{\Iso}{\operatorname{Iso}}
 \nc{\iso}{\mathfrak{iso}}
 \nc{\sso}{\mathfrak{so}}
\nc{\Ad}{\operatorname{Ad}} 
\nc{\Sym}{\mathrm{Sym}}
\nc{\Hol}{\mathrm{Hol}}

  \nc{\pr}{\operatorname{pr}} 
 \nc{\Dera}{\operatorname{Dera}} \nc{\Auto}{\operatorname{Auto}}
 \nc{\LL}{{\rm L}}
\nc{\dd}{{\rm d}}
\nc{\Id}{{\rm Id}}

\renewcommand{\d}{\mathrm{d}}

\newcommand{\mg}{\mathfrak g }

\newcommand{\mn}{\mathfrak n }

\newcommand{\mh}{\mathfrak h }

\renewcommand{\sl}{\mathfrak{sl} }

\newcommand{\lela}{ \left\langle}
\newcommand{\rira}{\right\rangle}
\newcommand{\bil}{\lela\cdot,\cdot\rira}

 \newtheorem{teo}{Theorem}[section]
 \newtheorem{pro}[teo]{Proposition}
 \newtheorem{cor}[teo]{Corollary}
 
 \newtheorem{defi}[teo]{Definition}

 \theoremstyle{definition}
 \newtheorem{ex}[teo]{Example}
  
 \newtheorem{remark}[teo]{Remark}

\usepackage[textwidth=1.7cm,textsize=scriptsize]{todonotes}
\usepackage{comment}

\allowdisplaybreaks
\date{}

\renewcommand{\d}{\mathrm{d}}

\renewcommand{\sl}{\mathfrak{sl} }

  \renewcommand{\H}{\mathbb H}

\newcommand{\tbfB}{\textbf{B}}

\newcommand{\tbfC}{\textbf{C}}

\def\w{\wedge}
\def\R{\mathbb{R}}
\def\C{\mathbb{C}}
\def\CP{\mathbb{C}\mathbb{P}}

\def\Lm{\Lambda}
\def\om{\omega}
\def\Om{\Omega}

\def\op{\oplus}

\def\U{\mathrm{U}}
\def\SO{\mathrm{SO}}
\def\SU{\mathrm{SU}}

\def\G{\mathrm{G}}

\def\ip{\raise1pt\hbox{\large$\lrcorner$}\>}

\numberwithin{equation}{section}
\newcommand \blue{\color{blue}}

\newcommand\gray{\color{lightgray}}
\newcommand\magenta{\color{magenta}}
\usepackage{titling}
\usepackage{changepage}

\usepackage{authblk}

\begin{document}

\title{Isometric solutions to the heterotic $\mathrm{G}_2$-system}

\author{Viviana del Barco, \quad Udhav Fowdar, \quad Andr\'es J. Moreno}

\maketitle

\vspace{-1.0cm}	

\begin{abstract}
In this note, we construct new solutions to the heterotic $\mathrm{G}_2$-system with non-abelian gauge group, both compact and non-compact, on certain $2$-step nilmanifolds and $3$-Sasakian manifolds. Our approach is based on an ansatz that allows us to vary both the $\mathrm{G}_2$-structure and the gauge data while keeping the underlying metric and orientation fixed. This leads, in particular, to distinct isometric solutions on the same manifold but with different gauge groups, and in some cases the resulting connection coincides with the characteristic connection of the $\mathrm{G}_2$-structure. We also investigate an $S^1$-invariant construction that yields further isometric solutions and with varying cosmological constant. Our results recover and extend several known examples solving the heterotic $\mathrm{G}_2$-system within a unified framework.
\end{abstract}

\section{Introduction}

Originating in theoretical physics, the heterotic $\mathrm{G}_2$-system (or $\rmG_2$-Hull-Strominger system) arises in the study of compactifications and domain wall solutions of heterotic string theory on $7$-dimensional manifolds \cites{ Fernandez2011, delaOssa2018, Clarke2022}. This can be viewed as the $\rmG_2$-analogue of the Hull-Strominger system on Calabi-Yau manifolds \cites{Str86,hull1986}.
In recent years, there has been a growing interest, in both mathematics and theoretical physics, in finding solutions to this system and understanding their moduli spaces cf. \cites{ fino2026,magdalena2020,dlOG21,OssaLaSv15,garcia2025parabolic,Iva10,Lotay2022,Lotay2024}. The main goal of this paper is to construct new solutions to the heterotic $\mathrm{G}_2$-system on certain $2$-step nilmanifolds and $3$-Sasakian manifolds by simultaneously varying the underlying $\mathrm{G}_2$-structure and the gauge group $G$. Before stating our results more precisely, we first recall the basic setup.

Let $(M,\varphi)$ denote a $7$-manifold endowed with a $\mathrm{G}_2$-structure determined by the $3$-form $\varphi$. Consider a principal $G$-bundle $P\to M$ with a connection $1$-form $A$ and let $\langle \cdot ,\cdot \rangle_{\mathfrak{g}}$ be an $\Ad(G)$-invariant non-degenerate symmetric bilinear form on its Lie algebra $\mathfrak{g}$ (shortly, $\mathrm{ad}$-invariant). Following \cite{Lotay2024}*{Definition 3.4}, we say that $(\varphi,A)$ satisfies the \emph{heterotic $\G_2$-system} if the following system of equations holds:
\begin{align}
\d\star \varphi &=4\tau_1\wedge \star \varphi,\label{equ: heterotic system 1}\\
F_{A}\wedge \star \varphi &=0,\label{equ: heterotic system 2}\\
\d T_{\varphi} &=\langle F_{A}\wedge F_{A}\rangle_{\mathfrak{g}} ,\label{equ: heterotic system 3}
\end{align}
where $\star$ denotes the Hodge star operator defined by $\varphi$, $\tau_1:=\frac{1}{12}\star(\varphi \w \star \d\varphi)$ is the intrinsic torsion $1$-form and $T_{\varphi}$ is the $3$-form defined by \eqref{equ: definition of T}. More geometrically, $T_{\varphi}$ can be defined as the torsion of the \textit{characteristic connection}: the unique $\mathrm{G}_2$ connection with skew-symmetric torsion, whose existence is equivalent to \eqref{equ: heterotic system 1} \cite{Friedrich2001}. Equation \eqref{equ: heterotic system 2} is the condition that $A$ is a $\mathrm{G}_2$-instanton. Equation \eqref{equ: heterotic system 3} is called the `heterotic Bianchi identity', and it couples the latter geometric and gauge-theoretic data into an intricate nonlinear constraint. 

To the best of our knowledge, the first non-trivial solutions to \eqref{equ: heterotic system 1}-\eqref{equ: heterotic system 3} were obtained in \cite{Fernandez2011} on certain Heisenberg nilmanifolds. Here by `non-trivial', we mean that the underlying $\G_2$-structure is not torsion free. More recently, new examples have been constructed on torus fibrations over K3 orbifolds in \cite{fino2026}, and in \cite{MRV} solutions on $2$-step nilmanifolds with abelian gauge groups were  classified. Approximate solutions have also been obtained in \cite{Lotay2022} on contact Calabi-Yau manifolds and in \cite{GS24} on certain $3$-$(\alpha,\delta)$-Sasakian manifolds with an exact solution occurring in the degenerate case. 

In the present work, we investigate two particular ans\"atze. The first ansatz applies whenever $M$ admits an $\SO(4)$-structure. Fixing a metric and orientation on $M$, there is a natural $\SO(3)$-family of compatible $\G_2$-structures. We vary the $\mathrm{G}_2$ form $\varphi$ in the latter $\SO(3)$-family and consider a connection $1$-form $A$ with $3$-dimensional gauge group $G$. Conditions \eqref{equ: heterotic system 1} and \eqref{equ: heterotic system 2} then lead to a coupling of the $\SO(3)$-variation and the gauge group $G$. When the underlying manifold $M$ is a $2$-step nilmanifold (belonging to a certain family) and $G=\mathbb{T}^3$ or $\SU(2)$, the heterotic Bianchi identity \eqref{equ: heterotic system 3} can be solved by choosing a suitable $\mathrm{ad}$-invariant pairing on $\mathfrak{g}$. On the other hand, for $G=\mathrm{SL}(2,\R)$, we need to supplement an additional $\mathrm{U}(1)$-connection in order to solve \eqref{equ: heterotic system 3}. Our precise results are contained in Proposition \ref{pro:del0}, \ref{pro:delneq0} and \ref{pro:delneq0A4}, and Corollary \ref{cor:del0}, \ref{cor:delneq0} and \ref{cor:delneq0A4}. We can roughly summarise them into the following:
\begin{teo}\label{teo:nilexistence}
	 Let $M=\Gamma\backslash N$ be a nilmanifold with nilpotent Lie algebra $\mathfrak{n}=\mathrm{Lie}(N)$ and $\Gamma$ be a co-compact lattice. There exists a co-closed  $\G_2$-structure $\varphi$ on $M$ and $G$-connection $A$ solving the heterotic $\G_2$-system \eqref{equ: heterotic system 1}-\eqref{equ: heterotic system 3} in the following cases:
	\begin{enumerate}
		\item $G=\mathbb{T}^3$ and $\mathfrak{n}\cong \R^2\oplus\mh_5$, $\R \oplus \mh_3^\C$, $\mh_{\H}$,
		\item $G=\mathrm{SU}(2)$ and $\mathfrak{n}\cong \R\oplus\mn_{3,2}$,
		$\mn_{7,3,A}$, 
		$\mn_{7,3,B_1}$,
		$\mn_{7,3,C}$,
		$\mh_\H$,
		\item $G=\mathrm{SL}(2,\R)\times \U(1)$ and $\mathfrak{n}\cong \R\oplus\mn_{3,2}$,
		$\mn_{7,3,A}$, 
		$\mn_{7,3,B_1}$,
		$\mn_{7,3,C}$,
		$\mh_\H$,
	\end{enumerate}
	where the structure constants for each $\mathfrak{n}$ is given in Remark \ref{rem:isoclass} below. 
	Moreover, in the non-abelian cases, both the underlying metric and volume form are the same but $\varphi$ are distinct. 
 \end{teo}
  
It turns out that for $\mathfrak{n}\cong\R \oplus \mathfrak{n}_{3,2}$ and $\mathfrak{h}_\H$, the above $\SU(2)$-connection $A$ can be identified with the characteristic connection of the underlying $\G_2$-structure, see Remark \ref{rem: relation to characteristic connection}. 
It was shown in \cite{CdBM} that any $2$-step nilpotent Lie algebra with a co-closed $\G_2$-structure whose characteristic connection is a $\G_2$-instanton is necessarily one of $\mathfrak{n}\cong \R \oplus \mathfrak{n}_{3,2}$, $\R^2 \oplus \mathfrak{h}_5$, $\mathfrak{h}_7$ and $\mathfrak{h}_\H$. It is rather striking that in all these cases, they also solve the heterotic Bianchi identity \eqref{equ: heterotic system 3}. The first such examples were found in \cite{Fernandez2011} (whereby the holonomy algebra is in fact abelian, see Example \ref{ex: more characteristic connection} below) and the remaining cases follow from our results here (in which case the holonomy algebra is instead $\mathfrak{su}(2)$). For the quaternion Heisenberg Lie algebra $\mathfrak{h}_{\H}$, this $\SU(2)$ example coincides with the exact solution found in \cite{GS24}. Remarkably, $\mathfrak{h}_\H$ admits distinct solutions to the heterotic $\G_2$-system with gauge group $\mathbb{T}^3$, $\SU(2)$ and $\SL(2,\R)\times \U(1)$. To the best of our knowledge, our $\SL(2,\R)\times \U(1)$ examples are the first known solutions to \eqref{equ: heterotic system 1}-\eqref{equ: heterotic system 3} with a non-compact, non-abelian gauge group.

Secondly, we investigate our $\SO(3)$ ansatz in the case when $M$ is a $3$-Sasakian manifold. Recall that a $3$-Sasakian manifold admits two distinct nearly parallel $\G_2$-structures: one induced from its standard $3$-Sasakian structure and one obtained by squashing the metric in the direction of the Reeb foliation \cites{FriedrichNP, Galicki1996}. In this case, to obtain new solutions of the heterotic $\mathrm{G}_2$-system, we must also introduce additional connections alongside our ansatz $G$-connection. To this end, we specialise to the $7$-sphere $S^7$ and the Aloff-Wallach space $N^{1,1}$. The additional connections are then obtained via pullback from the base quaternion-K\"ahler manifolds $S^4$ and $\CP^2$; these connections have gauge groups $\SU(2)$ and $\U(1)$, respectively, and are induced by the anti-self-dual part of the Levi-Civita connection \cite{AtiyahASD}. We denote by $\varphi_{ts}$ the $3$-Sasakian nearly parallel $\G_2$-structure, by $\varphi_{np}$ the squashed nearly parallel $\G_2$-structure, and by $\widehat{\varphi}_{ts}$ a co-closed (but not nearly parallel) $\G_2$-structure related to $\varphi_{ts}$ by our $\SO(3)$ ansatz, see \eqref{eq: 3sasa g2 structure modified} the definition. In the above notation, our result can be summarised as follows:
\begin{teo}\label{teo:3sasaki-unified}
Let $M$ be either $S^7$ or $N^{1,1}$. For the $\G_2$-structure $\varphi_{ts}$, there exist $G$-connections solving the heterotic $\G_2$-system in the following cases:
        \begin{enumerate}
            \item $G=\SL(2,\R)\times \SU(2)\times \G_2$ and $M=S^7$,
\item $G=\SL(2,\R)\times \U(1)$ and $M=N^{1,1}$.
            \end{enumerate} For the $\G_2$-structures $\widehat{\varphi}_{ts}$ and $\varphi_{np}$, there exist $G$-connections solving the heterotic $\G_2$-system in the following cases:
            \begin{enumerate}
            \item $G=\SU(2)\times \SU(2)$ and $M=S^7$,
            \item $G=\SU(2)\times \U(1)$ and $M=N^{1,1}$.
        \end{enumerate}   
\end{teo}
The above solutions are described in more detail in Example \ref{example: solution on S7}, \ref{example: solution on N11} and \ref{example: ts twisted}. For the solution on $(S^7, \varphi_{ts})$ with gauge group $\mathrm{SL}(2,\R) \times \mathrm{SU}(2)\times \mathrm{G}_2$, there is a free parameter in \eqref{equ: heterotic system 3} allowing us to include the example obtained in \cite{Ivanov2005}*{\S 6.1} using the characteristic connection (with gauge group $\G_2$, see Example \ref{example: ivanov-ivanov}). Note that in Theorem \ref{teo:3sasaki-unified}, only when the gauge group contains $\SL(2,\R)$ for $\varphi_{ts}$, we allow the pairing on $\mathfrak{g}$ to be only left-invariant, not $\mathrm{ad}$-invariant; otherwise, all pairings in this paper are always $\mathrm{ad}$-invariant. 

The supersymmetric compactification of heterotic supergravity on $M$ yields a $3$-dimensional manifold, either Minkowski or anti-de Sitter spacetime. This depends on whether the cosmological constant $\lambda=\tfrac{7}{12}\tau_0$ is zero or non-zero, where $\tau_0:=\tfrac17\star(d\varphi \w \varphi)$. This motivates our second ansatz which involves considering an $S^1$-variation of the $\G_2$-structure. In this case, we view $M$ as an $S^1$-bundle over a $6$-manifold $Q$ endowed with an $S^1$-family of $\SU(3)$-structures obtained by rotating the complex $(3,0)$-form. This generalises the ansatz considered in \cites{FinoG2T2023, hera2026} in the context of strong $\G_2$-structures with torsion. The upshot is that this allows us to vary the torsion forms $\tau_0$ and $\tau_1$ while again keeping the metric and orientation fixed, see Corollary \ref{cor:tforms}. In Theorem \ref{thm: S1- solutions}, we give sufficient conditions to lift solutions of the heterotic $\SU(3)$-system on $Q$ (or Hull-Strominger system when $Q$ is complex) to solutions of the heterotic $\G_2$-system on $M$ with an $S^1$-family of non-equivalent $\G_2$-structures $\varphi_t$ with different cosmological constant $\lambda$.

We give applications of Theorem \ref{thm: S1- solutions} in a few explicit examples. In Example \ref{ex: on s3s3}, we use an almost Bismut Hermitian Einstein $\SU(3)$-structure on $Q=S^3 \times S^3$ to construct an abelian solution on $M=S^3\times S^3\times S^1$. In Example \ref{ex: n32 again}, we apply the $S^1$-ansatz to the $\SU(2)$ solution obtained from Theorem \ref{teo:nilexistence} on the nilmanifold $ S^1 \times \Gamma  \backslash  N_{3,2}$, where $\mathrm{Lie}(N_{3,2})=\mathfrak{n}_{3,2}$.
As already mentioned above,  the latter solution in fact induces the characteristic connection. The $S^1$-ansatz yields a family of $\G_2$ $3$-forms $\varphi_t$ satisfying \eqref{equ: heterotic system 1}
but not co-closed in general, and still solving the heterotic $\G_2$-system. In particular, the latter family includes the non co-closed solution found in \cite{Ivanov2005}*{\S 6.2}. Lastly in Example \ref{ex: more characteristic connection}, we extend the original solutions found in \cite{Fernandez2011} on certain Heisenberg nilmanifold
to more general examples.

It is worth pointing out that both of our ans\"atze are applicable to more general situations. For instance, the $\SO(3)$ ansatz can be applied to general Lie groups, not just $2$-step nilpotent ones, and one can also consider general Aloff-Wallach spaces $N_{k,p}$ together with the $\G_2$-instantons found in \cite{BallGoncalo}, among other $7$-manifolds with $\SO(4)$-structures. Theorem \ref{thm: S1- solutions} can be generalised to allow for weaker conditions on the torsion of the $\SU(3)$-structure to obtain $S^1$-invariant solutions to \eqref{equ: heterotic system 1}-\eqref{equ: heterotic system 3}, see Proposition \ref{prop: S1- solutions}. 
Despite the special nature of our ans\"atze, it is striking that they recover most of the known solutions of the heterotic $\mathrm{G}_2$-system.

The outline of the paper is as follows: Section \ref{sec: preliminaries} contains the basics on $\mathrm{SU}(3)$- and $\mathrm{G}_2$-structures. The first ansatz is developed in Section \ref{sec:so3fam}, where Theorems \ref{teo:nilexistence} and \ref{teo:3sasaki-unified} are proved. Section \ref{sec: s1 invariant} contains the second, $S^1$-invariant ansatz, yielding solutions arising from $\mathrm{SU}(3)$-structures with special torsion.

\smallskip 

\noindent {\bf Acknowledgements:} This research was supported by MATHAMSUD Regional Program 24-MATH-12. VdB is partially supported by the S\~ao Paulo Research Foundation (Fapesp) grant [2024/19272-5]. UF was partially supported by Fapesp grant [2023/12372-1] during this project. AM was funded by Fapesp grant [2021/08026-5] and BRIDGES ANR--FAPESP: ANR-21-CE40-0017. The authors are grateful to Mario Garcia-Fern\'andez for helpful discussions on the topic. 


\section{Preliminaries}\label{sec: preliminaries}

In this section, we gather basic facts about $\mathrm{G}_2$- and $\mathrm{SU}(3)$-structures that will be used throughout the article and fix our conventions. Further details can be found in the standard references \cites{Bryant2003, Salamon1989}.

\subsection{Background on \texorpdfstring{$\rmG_2$}{}-structures} 
Let $(M,\vp)$ denote a $7$-manifold endowed with a $\rmG_2$-structure determined by the $3$-form $\vp$. Using $\vp$, we define a Riemannian metric $g_{\vp}$ and volume form $\vol_\vp$  on $M$ by
\begin{equation}
    6 g_\varphi(X,Y) {\rm vol}_\varphi = (X \lrcorner \varphi) \w (Y \lrcorner \varphi) \w \varphi, \label{def of metric and vol}
\end{equation}
where $X,Y$ are arbitrary vector fields and $\lrcorner$ denotes contraction. We denote the associated Hodge star operator by $\star$ and write $\psi:=\star \varphi$ for the dual $4$-form. As $\rmG_2$ modules, the space of differential forms $\Lm^\bullet(M)$ decompose into irreducible representations:
\begin{align*}
\Lm^1(M) &= \Lm^1_7,\\
\Lm^2 (M) &= \Lm^2_7 \oplus \Lm^2_{14}\\
\Lm^3(M) &= \langle \vp \rangle \oplus \Lm^3_7 \oplus \Lm^3_{27},
\end{align*}
where the subscript denotes the dimension of the irreducible module. We get the corresponding splitting for $\Lm^4$, $\Lm^5$ and $\Lm^6$ using the Hodge star operator. The above spaces can be explicitly defined by:
\begin{align}
\Lm^2_7&=\{\al\in \Lm^2\ |\ \star(\al \w \vp)=+2\al\},\nonumber\\
\Lm^2_{14}&=\{\al\in \Lm^2\ |\ \star(\al \w \vp)=-\al\},\label{equ: lambda214 1}\\
&=\{\al\in \Lm^2\ |\ \al\w \psi=0\},\label{equ: lambda214 2}\\
\Lm^3_7&=\{\star(\al\w \vp)\ |\ \al \in \Lm^1\},\nonumber\\
\Lm^3_{27}&=\{\al\in \Lm^3\ |\ \al \w \vp=0\ \text{\ and \ }\al \w \psi=0 \}.\nonumber
\end{align}
Following \cite{Bryant2003}, the $\rmG_2$ torsion forms $\tau_i$ are defined by
\begin{align}
    \d\vp &= \tau_0 \psi + 3 \tau_1 \w \vp + \star \tau_3, \label{equ: g2 torsion 1}\\
    \d\psi &= 4 \tau_1 \w \psi + \tau_2 \w \vp,\label{equ: g2 torsion 2}
\end{align}
where $\tau_0\in C^\infty(M)$, $\tau_1 \in \Lm^1$, $\tau_2 \in \Lm^2_{14}$ and $\tau_3 \in \Lm^3_{27}$.   

\begin{defi}
    If $\tau_2=0$, i.e. \eqref{equ: heterotic system 1} holds, then the $\rmG_2$-structure defined by $\vp$ is said to be `integrable' or  `$\rmG_2$ with torsion'. If, in addition, $\tau_1=0$ i.e. $d\psi=0$, then the $\rmG_2$-structure is called `co-closed'. 
\end{defi}

It was shown in \cite{Friedrich2001}*{Theorem 4.7} that $\tau_2=0$ if and only if there exists a $\rmG_2$ connection, i.e a connection preserving $\vp$, with totally skew-symmetric torsion $T_{\vp}$. Moreover, this connection is unique; this is called the {\em characteristic connection} and we shall denote it by $\nabla^c$. 
Its torsion $3$-form is explicitly given by
\begin{align}
\begin{split}
     T_{\vp} &= \frac{1}{6} \star (\d\vp \w \vp) \vp - \star \d\vp + \star(4\tau_1 \w \varphi) \label{equ: definition of T}\\
    &= \frac{1}{6} \tau_0 \vp + \star (\tau_1 \w \varphi) - \tau_3. 
\end{split}
\end{align}
Consider a principal $G$-bundle $P\to (M,\vp)$, or an associated vector bundle $E\to (M,\vp)$, with connection $1$-form $A$. 
We say that $A$ is a $\rmG_2$-instanton if its curvature form $F_A:=\d A+\frac{1}{2}[A\w A]$ lies in $\Omega^2_{14}(\mathrm{ad}P)$, or equivalently, in $\Omega^2_{14}(\mathrm{End}(E))$. From \eqref{equ: lambda214 2}, this is equivalent to the condition:
\begin{equation}
    F_A \wedge \psi =0. \label{equ: definition of instanton}
\end{equation}
This definition was introduced in \cite{Carrion1998}, generalising the
notion of anti-self-dual instanton in dimension $4$ cf. \cite{AtiyahASD}. 

\subsection{Background on \texorpdfstring{$\mathrm{SU}(3)$}{}-structures}  \label{sec:backsu3}
An $\SU(3)$-structure on a $6$-manifold $Q$ is given by a tuple $(g_{\omega},J,\omega,\Upsilon)$, where $g_{\omega}$ is a Riemannian metric, $J$ is a compatible almost complex structure, $\omega=g_\omega(J\cdot,\cdot)$ is the K\"ahler $2$-form and $\Upsilon$ is a complex $3$-form such that
\begin{equation*}
\omega \wedge \Upsilon=0,\qquad
\frac{i}2\Upsilon\wedge\overline{\Upsilon}=\frac23 \omega^3=4{\rm vol}_{\omega},
\end{equation*}
where $\omega^3=\omega \w \omega \w \omega$.
We denote by $\Upsilon_+$ and $\Upsilon_-$  the real and imaginary parts of $\Upsilon$, respectively. In \cite{Hitchin2000} Hitchin showed that the pair $(\omega, \Upsilon_+)$ determines the entire $\mathrm{SU}(3)$-structure, so we shall simply denote the $\SU(3)$-structure on $Q$ by the pair $(\omega, \Upsilon_+)$. 

Analogous to the $\G_2$ case, the space of differential forms $\Lambda^{\bullet}(Q)$ decompose into $\SU(3)$ irreducible modules:
\begin{align}
\begin{split}
\Lm^2 (Q) &= \langle\omega\rangle \oplus \Lm^2_{6}\oplus \Lm^2_8,\label{eq:L2su3}\\
\Lm^3(Q) &= \langle\Upsilon_+\rangle \oplus \langle\Upsilon_-\rangle \oplus \Lm^3_{6}\oplus \Lm^3_{12},
\end{split}
\end{align}
where each of the above irreducible modules can be characterised as follows:
\begin{align*}
\Lm^2_{6}&=\{\alpha\in \Lm^2\mid J\alpha=-\alpha\},\\
\Lm^2_8 &=\{\alpha\in\Lambda^2\mid J\alpha=+\alpha,\; \alpha\wedge \omega^2=0\},\\
\Lm^3_6&=\{\alpha\wedge \omega\mid \alpha\in \Lm^1\},\\
\Lm^3_{12}&=\{\gamma\in \Lm^3\mid \gamma\wedge \omega=0, \, \gamma\wedge \Upsilon_{\pm}=0\}.
\end{align*}
Given a $k$-form $\alpha$, we shall write 
$(\alpha)^k_l$ for its projection to $\Lambda^k_l$. Note that $\Lm^2_6$ is the real vector space underlying the space of complex $2$-forms of type $(2,0)+(0,2)$, and similarly $\Lm^2_8$ underlies the space of $2$-forms of type $(1,1)$ which are orthogonal to $\langle\omega\rangle$.

Following \cites{BeVe07, ChSa02}, the $\SU(3)$ torsion forms $\pi_i, \sigma_i, \nu_i$ are defined by
\begin{eqnarray}
\d \omega &=& -\frac32 \sigma_0 \Upsilon_++\frac32 \pi_0\Upsilon_-+\nu_1\wedge \omega+\nu_3,\nonumber\\
\d \Upsilon_+&=&\pi_0\omega^2+\pi_1\wedge \Upsilon_+-\pi_2\wedge \omega,\label{eq:torsu3}\\
\d \Upsilon_-&=&\sigma_0\omega^2+\pi_1\wedge \Upsilon_--\sigma_2\wedge \omega,\nonumber
\end{eqnarray}
where $\sigma_0,\pi_0\in C^\infty(Q)$, $\pi_1,\nu_1\in \Lm^1$, $\pi_2,\sigma_2\in \Lm^2_8$ and $\nu_3\in \Lm^3_{12}$. The underlying almost complex structure $J$ is integrable if and only if the torsion forms $\pi_0,\sigma_0,\pi_2,\sigma_2$ all vanish. 

Analogous to the $\mathrm{G}_2$ case,  the existence of a connection preserving $(\omega,\Upsilon_+)$ with totally skew-symmetric torsion is equivalent to $\pi_2=\sigma_2=0$ and $\pi_1=2\nu_1$, see \cite{Ivanov2005}*{Theorem 4.1}. The associated torsion tensor $T_{\omega}$ is then explicitly given by 
\begin{align}
    \begin{split}\label{eq: def T_omega}
        T_\omega:=&\ J(d\omega) - \hat{N}_J\\
                 =& \ \frac{\pi_0}{2}\Upsilon_++\frac{\sigma_0}{2}\Upsilon_-+J\nu_1\wedge\omega+J\nu_3,
    \end{split}
\end{align}
where $\hat{N}_J:=-2\pi_0\Upsilon_+-2\sigma_0 \Upsilon_-$ denotes the skew-symmetric part of the Nijenhuis tensor (after lowering the index using $g_{\omega}$). This connection  is also unique, and we shall call it as the \emph{Bismut connection} \cite{Bi89}.

As above, a connection $1$-form $A$ { with values in $\mg$} is called an \emph{$\SU(3)$-instanton} (or traceless Hermitian Yang-Mills) if its curvature $2$-form $F_A$ satisfies
\begin{equation}\label{eq: SU3-instanton}
    F_A\wedge \omega^2=F_A\wedge\Upsilon_+=0.
\end{equation}
In terms of the decomposition \eqref{eq:L2su3}, condition \eqref{eq: SU3-instanton} means that the $2$-form part of the curvature lies in $\Omega^2_8$. 
In analogy to \eqref{equ: heterotic system 1}-\eqref{equ: heterotic system 3}, we say that $(\omega,\Upsilon_+,A)$ satisfies the \emph{heterotic $\SU(3)$-system} if 
$\pi_2=\sigma_2=0$, $\pi_1=2\nu_1$, $A$ is an $\SU(3)$-instanton and the following $\SU(3)$ heterotic Bianchi identity holds:
\begin{align}\label{eq: het_SU3-Bianchi-ident}
    \d T_\omega = \langle F_A\wedge F_A\rangle_{\mathfrak{g}},
\end{align} for some bi-invariant pairing $\langle\cdot,\cdot\rangle_{\mathfrak{g}}$ on $\mg$.
When the underlying manifold is complex, i.e. $\hat{N}_J=0$, this is often called the Strominger system in the literature \cite{Str86}.
There is a natural relation between $\G_2$- and $\SU(3)$-structures provided by $S^1$-bundle constructions which we shall consider in Section \ref{sec: s1 invariant}.


\section{
\texorpdfstring{$\mathrm{SO}(3)$}{}-family of integrable \texorpdfstring{$\rG_2$}{}-structures}
\label{sec:so3fam}

\subsection{A general ansatz}

Let $(M^7,g,\mathrm{vol})$ denote an oriented Riemannian spin manifold. It is well-known that any $\mathrm{G}_2$-structure compatible with $(g,\mathrm{vol})$ is determined by a section of an $\mathbb{RP}^7\cong\mathrm{SO}(7)/\mathrm{G}_2$ bundle over $M$ cf.~\cite{Bryant2003}. If we now assume that $(M,g)$ admits a triple of orthogonal vector fields $\{e_5,e_6,e_7\}$, then we can distinguish an $\mathrm{SO}(3)\cong\mathbb{RP}^3$ family of compatible $\mathrm{G}_2$-structures in the latter $\mathbb{RP}^7$ family. Explicitly, any such compatible $\mathrm{G}_2$ $3$-form can be expressed as
\begin{gather}\label{eq:phians}
\varphi=\sigma_1^+\wedge E^5+\sigma_2^+\wedge E^6+\sigma_3^+\wedge E^7+E^{567},
\end{gather}
where $\{\sigma^+_1,\sigma^+_2,\sigma^+_3\}$ are self-dual $2$-forms on the transverse distribution $\langle e_5,e_6,e_7\rangle^{\perp}$ and
\begin{equation}\label{eq:Ebe}
\begin{pmatrix}
E^5 \\
E^6 \\
E^7 
\end{pmatrix}
:=
\textbf{B}  
\begin{pmatrix}
e^5 \\
e^6 \\
e^7 
\end{pmatrix},
\end{equation}
where $\textbf{B}$ denotes an $\mathrm{SO}(3)$-valued function on $M$. The matrix $\textbf{B}$ gives an explicit parametrisation of the aforementioned $\mathbb{RP}^3$ family. We can choose a local orthonormal co-framing $\{e^1,e^2,e^3,e^4\}$ of the transverse distribution so that $\{\sigma_i^+\}$ are given by
\begin{equation}\label{eq: def of sigmaplus}
	\sigma_1^+=e^{13}-e^{24}, \quad \sigma_2^+=-e^{14}-e^{23},\quad \sigma_3^+=e^{12}+e^{34}. 
\end{equation}
One checks directly using \eqref{def of metric and vol}:
\begin{gather}
g_{\varphi}=(e^1)^2+(e^2)^2+(e^3)^2+(e^4)^2+(e^5)^2+(e^6)^2+(e^7)^2,\\
\vol_{\varphi}=e^{1234567},\\
\psi=e^{1234}+\sigma_1^+\wedge E^{67}+\sigma_2^+\wedge E^{75}+\sigma_3^+\wedge E^{56}.\label{eq:psians}
\end{gather}
We emphasise here that $\psi=\star\varphi$ does depend on $\textbf{B}$ by \eqref{eq:Ebe}. 

Observe that from our hypothesis $M^7$ admits a natural $\mathrm{SO}(4)$-structure (this  is the subgroup of $\mathrm{G}_2$ preserving the distribution $\langle e_5,e_6,e_7\rangle$). Two important classes of such manifolds are given by Lie groups and 3-Sasakian manifolds, which we examine below.  

Consider now a real $3$-dimensional Lie group $G$ with associated Lie algebra $\mathfrak{g}$. We choose a basis $\{Y_5,Y_6,Y_7\}$ for $\mathfrak{g}$ such that
$
[Y_i,Y_j]= c_{ij}^k Y_k.
$
It will sometimes be convenient to use the matrix notation $\textbf{C}$, where  
\[ \textbf{C} = 
\begin{pmatrix}
c_{67}^5 & c_{75}^5 & c_{56}^5\\
c_{67}^6 & c_{75}^6 & c_{56}^6\\
c_{67}^7 & c_{75}^7 & c_{56}^7
\end{pmatrix}.
\]
On the trivial principal $G$ bundle $P=G \times M$, we define a natural connection $1$-form $A$ by 
\begin{equation}\label{eq:A}
A:=Y_5 e^5 + Y_6 e^6 + Y_7 e^7.
\end{equation}
It follows that its curvature $2$-form $F_A$ can be expressed as
\begin{equation} 
    F_A   = \sum_{i,j,k}Y_k(de^k+\frac{1}{2}c_{ij}^k e^{ij}).       \label{eq: FA}
\end{equation}
Using the vector notations $\textbf{e}:=(e^5, e^6,e^7)^T$ and $\textbf{e}^2:=(e^{67}, e^{75},e^{56})^T$, the instanton condition \eqref{equ: definition of instanton} can be expressed as
\begin{equation} 
(\d\textbf{e}+\textbf{C} \textbf{e}^2)
 \wedge \psi =0.\label{eq: instanton condition 2}
\end{equation}
where $\textbf{C} \textbf{e}^2$ should be understood as matrix multiplication. 
Furthermore, using the fact that $\textbf{B} \in \mathrm{SO}(3)$, one easily computes from \eqref{eq:Ebe}:
\begin{equation}
\textbf{E}^2
= \textbf{B} 
\textbf{e}^2,\label{eq: BE2}
\end{equation}
where following the above notation we write $\textbf{E}:=(E^5, E^6, E^7)^T $ and $\textbf{E}^2:=(E^{67}, E^{75}, E^{56})^T$.
In terms of $\textbf{E}$,
we can equivalently rewrite the instanton equation \eqref{eq: instanton condition 2} as
\begin{equation}
(\d\textbf{E} + 
\textbf{B}\textbf{C}\textbf{B}^T
\textbf{E}^2 ) \wedge \psi =0 
\label{eq: instanton condition 3}.
\end{equation}
\textbf{The setup:} The instanton condition \eqref{eq: instanton condition 3} gives a relation between (i) the structure of the gauge group $G$ determined by $\textbf{C}$, (ii) the choice of $\mathrm{G}_2$-structure on $M$ determined by $\textbf{B}$ and (iii) the structure equations of the underlying manifold determined by $de^{i}$ for $i=5,6,7$. Our goal is to find new solutions to the heterotic $\mathrm{G}_2$-system \eqref{equ: heterotic system 1}-\eqref{equ: heterotic system 3} by varying these three conditions, and additionally by choosing a suitable non-degenerate  symmetric bilinear form on the Lie algebra $\mathfrak{g}$. In view of this, we shall assume that $G$ is a reductive Lie group. Thus, $\mathfrak{g}$  is either abelian (i.e. $\R^3$), $\mathfrak{su}(2)$ or $\mathfrak{sl}(2,\R)$. In \cite{Mil76}*{\S 4} Milnor showed that one can always choose a basis of $\mg$ such that the matrix $\textbf{C}$ is diagonal i.e. $ \textbf{C} = {\rm diag}(\lambda_5,\lambda_6,\lambda_7)$. In this case, the $\mathrm{G}_2$-instanton condition \eqref{eq: instanton condition 2} reads:
\begin{equation}
    \big(de^5+ \lambda_5 e^{67}  \big)\wedge \psi = \big(de^6+ \lambda_6 e^{75}\big)\wedge \psi= \big(de^7+ \lambda_7 e^{56} \big) \wedge \psi=0.
   \label{eq: instanton condition 1}
\end{equation}

Fixing the choice of diagonal $\textbf{C}$ can be viewed as a gauge fixing condition for our connection form $A$. Depending on the signs of $\lambda_5,\lambda_6,\lambda_7$ we get different isomorphism classes for $\mg$: $\mg$ is abelian if all the $\lambda_i$ are zero,  $\mg=\su(2)$ if all the $\lambda_i$ are non-zero and all have the same sign, and $\mg=\sl(2,\R)$ if all the $\lambda_i$ are non-zero but do not all have the same sign. In the non-abelian cases, the (unique up to a constant factor) bi-invariant form on $\mathfrak{g}$ is given by $\langle\cdot ,\cdot \rangle_{\mathfrak{g}}={\rm diag}(\lambda_6 \lambda_7,\lambda_5\lambda_7,\lambda_5\lambda_6)$ with respect to the basis $\{Y_5,Y_6,Y_7\}$.


\subsection{\texorpdfstring{$2$}{}-step nilpotent Lie group case}

Motivated by the results in \cite{CdBM}, in this section we investigate our $\mathrm{SO}(3)$ ansatz on certain classes of $2$-step nilpotent Lie group. We shall assume that the underlying manifold $M$ is a nilpotent Lie group with Lie algebra $\mn$ and it admits a left-invariant co-framing $\{e^i\}_{i=1}^7$ (i.e. a basis of $\mn^*$) satisfying
the following structure equations:
\begin{equation}
d\begin{pmatrix}
e^{5}\\
e^{6}\\
e^{7}
\end{pmatrix} = \textbf{A}^+\begin{pmatrix}
\sigma_1^+\\
\sigma_2^+\\
\sigma_3^+
\end{pmatrix}+
\textbf{A}^-\begin{pmatrix}
\sigma_1^-\\
\sigma_2^-\\
\sigma_3^-
\end{pmatrix},\label{eq: str nil general}
\end{equation}
where $\textbf{A}^\pm$ are arbitrary $3 \times 3$ matrices, $\sigma_i^+$ are given in \eqref{eq: def of sigmaplus},
\begin{equation*}
	\sigma_1^-:=e^{13}+e^{24}, \quad \sigma_2^-:=e^{14}-e^{23},\quad \sigma_3^-:=e^{12}-e^{34},
\end{equation*}
denote anti-self-dual $2$-forms on $\langle e^1,e^2,e^3,e^4\rangle$, and $de^i=0$ for $i=1,2,3,4$. It follows that locally $M$ can viewed as a $\mathbb{T}^3$-bundle over $\mathbb{T}^4$. 
 First we consider the instanton condition \eqref{equ: heterotic system 2}:
\begin{pro} \label{pro:dAA}
Suppose that the structure equations \eqref{eq: str nil general} hold. Then the connection $A$ given by \eqref{eq:A}  is a $\mathrm{G}_2$-instanton with respect to the $\mathrm{G}_2$-structure induced by $\varphi$ given by \eqref{eq:phians} if and only if $\textbf{C}=-2\textbf{A}^+\textbf{B}$.
\end{pro}
\begin{proof}
    Using \eqref{eq: str nil general}, we can rewrite \eqref{eq: instanton condition 3} as
    \begin{equation}
\Big(\textbf{B}\textbf{A}^+\w \begin{pmatrix}
\sigma_1^+\\
\sigma_2^+\\
\sigma_3^+
\end{pmatrix} + 
\textbf{B}\textbf{C}\textbf{B}^T
\textbf{E}^2\Big ) \wedge \psi =0,
\end{equation}
where we used that $\sigma^-_i\wedge\psi=0$. Multiplying the above by $\textbf{B}^T$ on the left, it is not hard to see using expression \eqref{eq:psians} for $\psi$ that $A$ is a $\mathrm{G}_2$-instanton if and only if $-2\textbf{A}^+=\textbf{C}\textbf{B}^T$. This concludes the proof.
\end{proof}
Next we consider the integrability condition \eqref{equ: heterotic system 2}:
\begin{pro}\label{prop: instanton implies coclosed}
    If the connection $A$ given by \eqref{eq:A}  is a $\mathrm{G}_2$-instanton with respect to the $\mathrm{G}_2$-structure induced by $\varphi$ \eqref{eq:phians}, then $\varphi$ is co-closed.
\end{pro}
\begin{proof}
  Using \eqref{eq: BE2} and the structure equations \eqref{eq: str nil general}, we have
  \begin{align*}
  \d \mathbf{E}^2 &= \mathbf{B} \begin{pmatrix}
de^{6}\w e^7-e^6\w de^7\\
de^{7}\w e^5-e^7\w de^5\\
de^{5}\w e^6-e^5\w de^6
\end{pmatrix}
=\mathbf{B} 
\begin{pmatrix}
0 & +e^7 & -e^6\\
-e^7 & 0 & +e^5\\
+e^6 & -e^5 & 0
\end{pmatrix}\w
\begin{pmatrix}
de^{5}\\
de^{6}\\
de^{7}
\end{pmatrix}\\
&=\mathbf{B} 
\begin{pmatrix}
0 & +e^7 & -e^6\\
-e^7 & 0 & +e^5\\
+e^6 & -e^5 & 0
\end{pmatrix}\w\Big(
\textbf{A}^+\begin{pmatrix}
\sigma_1^+\\
\sigma_2^+\\
\sigma_3^+
\end{pmatrix}+
\textbf{A}^-\begin{pmatrix}
\sigma_1^-\\
\sigma_2^-\\
\sigma_3^-
\end{pmatrix}\Big).
\end{align*}
Writing $\kappa_1$ for the matrix of $1$-form consisting of $e^5,e^6,e^7$ 
occurring in the latter equation, 
we now compute using Einstein summation convention: 
\begin{align*}
    d\psi &= \sigma_i^+ \w d(\textbf{E}^{2})_{i}\\
    &= \sigma_i^+ \w(\textbf{B}_{ij} (\kappa_1)_{jk} \w (\textbf{A}^+)_{kp}\sigma^+_p)\\
    &=2(\textbf{B}_{ij} (\kappa_1)_{jk} \w (\textbf{A}^+)_{ki})e^{1234}\\
    &= -(\textbf{B}_{ij} (\kappa_1)_{jk} \w (\textbf{C})_{kq}(\textbf{B}^T)_{qi})e^{1234}\\
    &= - (\kappa_1)_{qk}\w (\textbf{C})_{kq}e^{1234} \\
    &=0,
\end{align*}
where we used $\sigma^+_i \w \sigma_j^-=0$ in the second line,  $\sigma^+_i \w \sigma_j^+=2\delta_{ij}e^{1234}$ in the third line, the instanton condition $\textbf{C}=-2\textbf{A}^+\textbf{B}$ in the fourth line, $\tbfB^T\tbfB= \mathrm{Id}$ in the fifth line and in the final line we used that $\kappa_1$ is skew-symmetric while $\textbf{C}$ is diagonal hence symmetric. This concludes the proof.
\end{proof}
It was shown in \cite{CdBM}*{Theorem 1.1} that for \textit{any} left-invariant co-closed $\mathrm{G}_2$-structure on a $2$-step nilpotent Lie algebra, the associated characteristic connection $\nabla^c$ is a $\mathrm{G}_2$-instanton implies the existence of a basis $\{e_i\}_{i=1}^7$ of $\mn$ such that the $\mathrm{G}_2$-structure is given by $\varphi$ as in \eqref{eq:phians} with $\textbf{B}=\mathrm{Id}$; moreover, $\textbf{A}^+$ has to be a multiple of the identity, see also Remark \ref{rem: relation to characteristic connection} below. Motivated by the latter, we shall henceforth assume 
\begin{equation}
    \textbf{A}^+=\delta\cdot \mathrm{Id}\qquad\text{and}\qquad \textbf{A}^-=\diag(\epsilon_1, \epsilon_3, \epsilon_3),\label{eq:strc}
\end{equation}
where $\delta,\epsilon_i\in \R$. 
\begin{remark}
    The assumption that $\textbf{A}^-$ is diagonal is not really a constraint here since one can always redefine the $2$-forms $\sigma^-_i$ while leaving $\varphi$ unchanged since $\mathrm{SU}(2)\subset\mathrm{G}_2$. The non-trivial hypothesis here is the choice of $\textbf{A}^+$, which one can indeed consider to be more general but we do not investigate this in the present work.
\end{remark}
Since $\textbf{C}$ is diagonal and $\textbf{B}\in \mathrm{SO}(3)$, without loss of generality we can take $\textbf{B} = \diag(1,a,a)$, where $a=\pm 1$. From Proposition \ref{pro:dAA}, the $\mathrm{G}_2$-instanton condition \eqref{equ: heterotic system 2} implies $\textbf{C}=-2\delta\cdot\diag(1, a,a)$. Thus, the gauge group $G$ is abelian when $\delta=0$, $\mathrm{SU}(2)$ when $a=+1$ and $\mathrm{SL}(2,\R)$ when $a=-1$. This shows that the choice of the gauge group (given by $\textbf{C}$) is dependent on the choice of $\mathrm{G}_2$-structure (given by $\textbf{B}$) via the $\mathrm{G}_2$-instanton condition. Furthermore, from Proposition \ref{prop: instanton implies coclosed} we also know $\varphi$ is co-closed hence \eqref{equ: heterotic system 1} holds. Thus, we only need to solve for the heterotic Bianchi identity \eqref{equ: heterotic system 3}.

\begin{remark} \label{rem:isoclass} Using Gong's classification of $7$-dimensional nilpotent Lie algebras \cite{Gon98}, one can list all the possible isomorphism classes of Lie algebras described by \eqref{eq:strc}. 

When $\delta=0$, a Lie algebra $\mathfrak{n}$ verifying \eqref{eq:strc} is isomorphic to one of the following:
\[
\R^7,\quad\R^2\oplus\mh_5,\quad  \R\oplus\mh_3^\C,\quad \mh_{\H},
\]
 where in the standard Salamon's notation \cite{SA1} the above Heisenberg Lie algebras can be described by:
\begin{align*}
    \mathfrak{h}_{2k+1} &= \big(0,...,0,12+...+(2k-1)(2k)\big),\\
    \mathfrak{h}_{3}^{\C}&= (0,0,0,0,12-34,13+24),\\ \mathfrak{h}_\H &= (0,0,0,0,12-34,13+24,14-23), 
\end{align*}
for $k \in \mathbb{N}$.  Indeed, these cases are distinguished by the vanishing pattern of the $\epsilon_i$: either all vanish, or exactly one, two, or three are non-zero.

When $\delta \neq 0$, a Lie algebra $\mathfrak{n}$ verifying \eqref{eq:strc} is isomorphic to one of the following:
\begin{align*}
     \mn_{6,3}\oplus\R &= (0,0,0,0,12,13,23)=\mn_{3,2}\oplus\R \\
     \mn_{7,3,A}&= (0,0,0,0,12,23,24)  \\
        \mn_{7,3,B_1}&= (0,0,0,0,13+23,12-34,14)  \\
     \mn_{7,3,C}&= (0,0,0,0,12+34,23,24)  \\  
\mn_{7,3,D_1}&= (0,0,0,0,12-34,13+24,14-23)=\mh_\H
\end{align*}
The notation on the left is the one used in \cite{Gon98}. On the right, for the first and last Lie algebra, we include the notation used in \cite{CdBM}. As shown in \cites{BFF18,dBMR}, all
these Lie algebras admit co-closed $\G_2$-structures. Observe that the quaternion Heisenberg Lie algebra $\mathfrak{h}_\H$ can occur in both cases with $\delta=0$ and $\delta\neq0$.
\end{remark}

Next we need to compute the torsion.
A direct calculation using \eqref{eq: str nil general} and \eqref{eq:strc} shows that $\tau_0=\frac{4}{7}\delta(2a+1)$  and the torsion $3$-form \eqref{equ: definition of T} is given by
\begin{align*}
T_{\varphi}=  &-\frac{4}{3}\delta(2a+1)e^{567}+\frac13\delta\left(4a-1\right)(e^5\wedge \sigma_1^+)+\frac13\delta\left(2a+1\right)(e^6\wedge\sigma^+_2+e^7\wedge\sigma^+_3)\\
 &+ \left(\epsilon_1 e^5\wedge\sigma^-_1+\epsilon_2 e^6\wedge\sigma^-_2+\epsilon_3 e^7\wedge\sigma^-_3\right).
 \end{align*}
In particular, we see that the $\mathrm{G}_2$-structure is purely co-closed, i.e. $\tau_0=0$, precisely if $\delta=0$. From the structure equations \eqref{eq:strc}, we also compute: 
\begin{align}\label{eq:dT}
 \begin{split}
     \d T_{\varphi} =   &-\frac{4}{3}\delta(2a+1)\left((\delta\sigma^+_1+\epsilon_1\sigma^-_1)\wedge e^{67}+(\delta\sigma^+_2+\epsilon_2\sigma^-_2)\wedge e^{75}\right.\\
     &+\left.(\delta\sigma^+_3+\epsilon_3\sigma^-_3)\wedge e^{56}\right)+2\left(\frac13(8a+1)\delta^2-\epsilon_1^2-\epsilon^2_2-\epsilon^2_3\right)e^{1234} .
 \end{split}
\end{align}
Next we consider the cases $\delta=0$ and $\delta\neq 0$ separately.\\
\noindent \textbf{The $\delta=0$ case.} In this case $\textbf{C}=0$ i.e. $G$ is abelian. Equation \eqref{eq:dT} simplifies to
\begin{equation}
\d T_{\varphi} = -2(\epsilon_1^2+\epsilon_2^2+\epsilon_3^2) e^{1234}.\label{equ dT delta0}
\end{equation}
Observe that $\d T_\varphi=0$ i.e. $\varphi$ is a strong $\mathrm{G}_2$-structure with torsion if and only if $M=\mathbb{T}^7$ is the flat torus; this corresponds to a trivial solution.
Since $G$ is abelian, any non-degenerate symmetric bilinear form on $\mathfrak{g}$ is $\mathrm{ad}$-invariant. We consider the diagonal pairing given by:
\begin{equation}\label{eq:pairing}
\left\langle Y_i,Y_j\right\rangle = ra_{ii}\delta_{ij},
\end{equation}
and hence we get
\begin{equation}
\left\langle F_A\wedge F_A\right\rangle=-2r (a_{11}\epsilon_1^2+a_{22}\epsilon_2^2+a_{33}\epsilon_3^2)e^{1234}.\label{equ fafa delta0}
\end{equation}
Comparing \eqref{equ dT delta0} and \eqref{equ fafa delta0}, we see that provided $(a_{11}\epsilon_1^2+a_{22}\epsilon_2^2+a_{33}\epsilon_3^2) \neq 0$, we can always solve \eqref{equ: heterotic system 3} for $r$ and hence get a solution to the heterotic $\mathrm{G}_2$-system. Note, however, that for any such solution, the pairing $\langle\cdot,\cdot\rangle$ on $\mathfrak{g}$ cannot be negative definite i.e. $ra_{ii}$ cannot all be negative. On the other hand, one can easily check that the pairing can be chosen to be either positive definite or of mixed signature: $(1,2)$ or $(2,1)$, whenever one of the $\epsilon_i$ is non-vanishing. We can summarise the above into:
\begin{pro}\label{pro:del0}
Let $\mn$ be a nilpotent Lie algebra with structure coefficients given by \eqref{eq:strc} with $\delta=0$. Consider on $\mn$ the $\mathrm{SO}(3)$-family of $\mathrm{G}_2$-structures $\varphi$ defined by \eqref{eq:phians} together with the connection $A$ defined by \eqref{eq:A} with gauge group $G$. Then $A$ is a $\mathrm{G}_2$-instanton with respect to $\varphi$ if and only if $G$ is abelian, in which case $\varphi$ is purely co-closed (for any $\textbf{B}\in \mathrm{SO}(3)$).

Moreover, for signatures either positive definite or mixed $(1,2)$ or $(2,1)$, there exists a $\mathrm{ad}$-invariant pairing $\bil$ on $\mg$ with the prescribed signature such that the curvature $F_A$ satisfies the heterotic Bianchi identity. In particular, this yields solutions to the heterotic $\G_2$-system \eqref{equ: heterotic system 1}-\eqref{equ: heterotic system 3}.
\end{pro}

From Remark \ref{rem:isoclass}, we see that each of the Lie algebras corresponding to $\delta=0$ admits a basis with rational structure constants. It follows that each of the associated simply connected nilpotent Lie group $N$ admits a co-compact lattice $\Gamma$ \cite{MAL}. Thus, the left invariant solutions provided by Proposition \ref{pro:del0} on $N$ descend to the quotient $M=\Gamma \backslash N$, and we have: 
\begin{cor} \label{cor:del0} Let $M=\Gamma\backslash N$  be a nilmanifold whose Lie algebra $\mn$ is isomorphic to one of the following:
\[
\R^2\oplus\mh_5,\quad \R \oplus \mh_3^\C, \quad \mh_{\H}.
\]
Then $M$ admits a purely co-closed $\mathrm{G}_2$-structure $\varphi$ and an abelian connection $A$ such that $(\varphi,A)$ is a solution to the heterotic $\G_2$-system \eqref{equ: heterotic system 1}-\eqref{equ: heterotic system 3}.
\end{cor}

\begin{remark}
The fact that negative definite signatures cannot occur for the solutions in Proposition \ref{pro:del0} follows more generally by \cite{Lotay2024}*{Theorem 3.9}. More precisely, since $g_{\varphi}$ is a left-invariant metric on a nilpotent Lie algebra, it has non-negative scalar curvature \cite{Mil76}*{Theorem 3.1}. As we also have $\tau_1=0$, from \cite{Lotay2024}*{(3.18)} it follows that in this case $|F_A|^2_{\mathfrak{g}}$ has to be non-negative in order to solve \eqref{equ: heterotic system 3}. 
\end{remark}
\noindent \textbf{The $\delta \neq 0$ case.} In this case, we recall that $\textbf{C}=-2\delta \cdot \diag(1,a,a)$ with $a=+1$ or $-1$ corresponding to ${G}=\mathrm{SU}(2)$ or $\mathrm{SL}(2,\R)$, respectively. Consider the diagonal pairing on $\mathfrak{g}$ given by 
\begin{equation*}
    \langle\cdot,\cdot \rangle_{\mathfrak{g}}=\alpha\cdot\diag(\gamma,1,1),
\end{equation*}
with respect to the basis $\{Y_5,Y_6,Y_7\}$, where $\gamma\in\{\pm 1\}$ and $\alpha\in \R^*$. This corresponds to a bi-invariant metric on $G$ precisely when $\gamma=a$. Using this pairing, a long but straightforward computation gives:
\begin{align} \label{eq:FAFA}
\begin{split}
    \langle F_A\wedge F_A\rangle_{\mathfrak{g}}=& \ -4\delta\alpha\left[\gamma(\delta\sigma_1^++\epsilon_1\sigma^-_1)\wedge e^{67}+a(\delta\sigma_2^++\epsilon_2\sigma^-_2)\wedge e^{75}\right.\\ 
    &\left.+a(\delta\sigma_3^++\epsilon_3\sigma^-_3)\wedge e^{56}\right] +2\alpha (\delta^2(\gamma+2)-\gamma\epsilon_1^2-\epsilon^2_2-\epsilon^2_3)e^{1234} .
    \end{split}
\end{align}
Comparing with \eqref{eq:dT} we have
 \begin{align}
    \d T_{\varphi} -\langle F_A\wedge F_A\rangle_{\mathfrak{g}} = & \  2\left[\frac{\delta^2}3\left( 8a+1-3\alpha(\gamma+2)\right)-\epsilon_1^2(1-\gamma\alpha)-(\epsilon^2_2+\epsilon^2_3)(1-\alpha) \right]e^{1234}\nonumber\\
    & -\frac4{3}\delta\left\{
((2a+1-3\alpha\gamma)\left((\delta+\epsilon_1)e^{1367} +(-\delta+\epsilon_1)e^{2467} \right)\right.\label{eq:Bans}\\
&+\left.((2a+1-3\alpha a)\left[(-\delta+\epsilon_2)e^{1475}+(-\delta-\epsilon_2)e^{2375}\right.\right.\nonumber\\
&\left.\left.+(\delta+\epsilon_3)e^{1256}+(\delta-\epsilon_3)e^{3456}\right]\right\}.\nonumber
\end{align}

Setting
\begin{equation}
    \gamma=a\qquad\text{and} \qquad \alpha=\frac{2a+1}{3a}
\end{equation}
in \eqref{eq:Bans}, the above simplifies to:
\begin{equation}\label{eq:Bans simplified}
    \d T_{\varphi} -\langle F_A\wedge F_A\rangle_{\mathfrak{g}} =  \frac{2(a-1)}{3}  \left[(6+2a)\delta^2+2\epsilon_1^2-a(\epsilon_2^2+\epsilon_3^2) \right]e^{1234}.
\end{equation}
It is not hard to see that $a=-1$ does not solve the latter. The only solutions occur when $a=\gamma = \alpha =1$ and $\epsilon_i \in \R$ are arbitrary; hence $G=\mathrm{SU}(2)$ in this case. We can summarise the above results into:
\begin{pro}\label{pro:delneq0}
Let $\mn$ be a nilpotent Lie algebra with structure coefficients given by \eqref{eq:strc} with $\delta\neq 0$. Consider on $\mn$ the $\mathrm{G}_2$-structure $\varphi$ defined by \eqref{eq:phians} with $\textbf{B}=\mathrm{Id}$ together with the connection $A$ defined by \eqref{eq:A} with gauge group $G=\mathrm{SU}(2)$ and $\textbf{C}=-2\delta \mathrm{Id}$. Then $\varphi$ defines a co-closed $\mathrm{G}_2$-structure (with $\tau_0 \neq 0$), $A$ is a $\mathrm{G}_2$-instanton with respect to $\varphi$, and $(\varphi,A)$ solves the heterotic Bianchi identity \eqref{equ: heterotic system 3} for the $\mathrm{ad}$-invariant pairing $\bil_{\mathfrak{g}}=-(8\delta^2)^{-1}\kappa$, where $\kappa$ denotes the Killing form.
\end{pro}
As before, using Remark \ref{rem:isoclass} one can check that the Lie algebras in the ansatz corresponding to $\delta\neq 0$ admit rational bases and thus co-compact lattices. This together with Proposition \ref{pro:delneq0} gives:
\begin{cor} \label{cor:delneq0} Let $M=\Gamma\backslash N$ be a nilmanifold whose Lie algebra $\mn$ is isomorphic to one of the following:
\[\R\oplus\mn_{3,2}, \quad
     \mn_{7,3,A}, \quad
     \mn_{7,3,B_1},\quad
     \mn_{7,3,C},\quad
\mh_\H.
\]
Then $M$ admits a co-closed $\mathrm{G}_2$-structure $\varphi$ and a connection $A$ with gauge group $\SU(2)$ such that $(\varphi,A)$ is a solution to the heterotic $\G_2$-system \eqref{equ: heterotic system 1}-\eqref{equ: heterotic system 3}.
\end{cor}
\begin{remark}\textbf{Relation to the characteristic connection.}\label{rem: relation to characteristic connection} \\
In \cite{CdBM}*{Theorem 1.1}, it was shown that the only $2$-step nilpotent Lie algebras with $3$-dimensional commutator possessing a co-closed  $\G_2$-structure for which the characteristic connection $\nabla^c$ is a $\G_2$-instanton are: 
\[
\R\oplus\mn_{3,2}
\quad\text{and}\quad \mh_\H.
\]
It turns out that in both of these cases the connection $A$ of Proposition \ref{pro:del0} induces the characteristic connection of the underlying co-closed  $\G_2$-structure via an embedding $A\in \Om^1(\mathfrak{su}(2))\hookrightarrow \Omega^1(\mathfrak{g}_2)$ as a $\G_2$-connection on the tangent bundle on $M$.

More explicitly, adapted to the $\G_2$ co-framing $\{e^1,...,e^7\}$ as above, the characteristic connection $\nabla^c$ on $\R\oplus\mn_{3,2}$ is given by:
{\small\[
2
\left(
\begin{array}{c|c|c}
0& 0&0\\
\hline
0& \begin{array}{ccc} 
 0&e^6&e^5\\
 -e^6&0&-e^7\\
 -e^5&e^7&0
\end{array}
&0\\
\hline
0&0&\begin{array}{ccc}
     0&e^7&-e^6\\
    -e^7&0&e^5\\
    e^6&-e^5&0
\end{array}
\end{array}
\right)
\]}
and the characteristic connection $\nabla^c$ on $\mh_\H$ is given by:
\small\[
\left(
\begin{array}{c|c}
\begin{array}{cccc}
   0&-e^7&-e^5&e^6\\
    e^7&0&e^6&e^5\\
    e^5&-e^6&0&-e^7\\
    -e^6&-e^5&e^7&0\\
\end{array}&0\\
\hline
0&\begin{array}{ccc}
 0&2e^7&-2e^6\\
 -2e^7&0&2e^5\\
 2e^6&-2e^5&0
\end{array}
\end{array}
\right)
.\]
In particular, the results in \cite{CdBM} also imply that the connection $A$ in Proposition \ref{pro:del0} {\em is not} the characteristic connection for the nilpotent Lie algebras $\mn\ncong  \R\oplus\mn_{3,2},\mh_\H$. In fact, \cite{CdBM}*{Theorem 1.1} classifies all the $2$-step nilpotent Lie algebras which admit co-closed $\G_2$-structures with $\nabla^c$ a $\G_2$-instanton, see Example \ref{ex: more characteristic connection} below for the remaining cases which instead have $\mathfrak{hol}(\nabla^c)\cong \R$ (in contrast to the above examples which have $\mathfrak{hol}(\nabla^c)\cong \mathfrak{su}(2)$).
\end{remark}
If instead we now set $a=-1$ in \eqref{eq:Bans simplified} then we have:
\begin{equation}
    \d T_{\varphi} -\langle F_A\wedge F_A\rangle_{\mathfrak{g}} =  -\frac{4}{3}  \left[4\delta^2+2\epsilon_1^2+\epsilon_2^2+\epsilon_3^2 \right]e^{1234}.\label{eq:Bans a=-1}
\end{equation}
In this case, $a=\gamma=-1$ and $\alpha=\frac{1}{3}$. Hence $\textbf{B}=\diag(1,-1,-1)$ and from Proposition \ref{pro:dAA} we have $\textbf{C}=2\delta\diag(-1,1,1)$. Thus, this corresponds to a distinct co-closed $\G_2$-structure than in Proposition \ref{pro:del0} (which is nonetheless isometric), the gauge group of the $\mathrm{G}_2$-instanton $A$ is now $G=\mathrm{SL(2,\R)}$ and $\langle\cdot,\cdot\rangle_{\mathfrak{g}}$ is again a multiple of the Killing form. 

We shall now supplement a second connection to correct for the heterotic Bianchi identity in \eqref{eq:Bans a=-1}. Consider a principal $\mathrm{U}(1)$-bundle over the nilpotent Lie group $M$ endowed with a connection $1$-form $\xi\otimes Y_0$, where $Y_0$ denotes a generator of $\mathfrak{u}(1)\cong\R$, and with curvature form $d\xi = \sigma_1^-$. If $\Gamma$ denotes a co-compact lattice in $M$ (which exists by \cite{MAL}), we can view $M/\Gamma$ as a $\mathbb{T}^3$-bundle over $\mathbb{T}^4$ with $[\sigma_1^-]\in H^2(\mathbb{T}^4,\mathbb{Z})$, see \cite{MRV}. Thus, this principal $\mathrm{U}(1)$-bundle is pullbacked from the base. It is clear that $\sigma_1^-\w \psi=0$ i.e. $\xi$ is a $\mathrm{G}_2$-instanton, and $\sigma_1^-\w\sigma_1^-=-2e^{1234}$. Hence we can define the connection $\hat{A}:=(\xi\otimes Y_0)\oplus A $ on a principal $\mathrm{U}(1)\times\mathrm{SL}(2,\R)$-bundle over $M$ with the $\mathrm{ad}$-invariant pairing on $\mathfrak{u}(1)$ given by $|Y_0|^2=\frac23(4\delta^2+2\epsilon_1^2+\epsilon_2^2+\epsilon^2_3)$
so that \eqref{eq:Bans a=-1} becomes:
\begin{eqnarray}
    d T_{\varphi} =\langle F_{\hat{A}}\wedge F_{\hat{A}}\rangle, 
\end{eqnarray}
i.e. this yields another solution to \eqref{equ: heterotic system 1}-\eqref{equ: heterotic system 3} distinct from Proposition \ref{pro:del0}. We summarise the above into:
\begin{pro}\label{pro:delneq0A4}
Let $\mn$ be a nilpotent Lie algebra with structure coefficients given by \eqref{eq:strc} with $\delta\neq 0$. Consider on $\mn$ the $\mathrm{G}_2$-structure $\varphi$ defined by \eqref{eq:phians} with $\textbf{B}=\diag(+1,-1,-1)$ together with the connection $A$ defined by \eqref{eq:A} with gauge group $G=\mathrm{SL}(2,\R)$ and $\textbf{C}=2\delta \diag(-1,+1,+1)$. Then $\varphi$ defines a co-closed $\mathrm{G}_2$-structure (with $\tau_0 \neq 0$) and $A$ is a $\mathrm{G}_2$-instanton with respect to $\varphi$. Additionally, there exists another $\mathrm{G}_2$-instanton $\xi$ on a principal $\mathrm{U}(1)$-bundle, so that the product connection $\hat{A}$ on the $\mathrm{U}(1)\times \mathrm{SL}(2,\R)$-bundle solves the heterotic Bianchi identity \eqref{equ: heterotic system 3} for a suitable $\mathrm{ad}$-invariant pairing on $\mathfrak{u}(1)\oplus\mathfrak{sl}(2,\R)$.
\end{pro}

 The following result follows from Remark \ref{rem:isoclass} and Proposition \ref{pro:delneq0A4}.
\begin{cor} \label{cor:delneq0A4} Let $M=\Gamma\backslash N$ be a nilmanifold whose Lie algebra $\mn$ is isomorphic to one of the following:
\[\R\oplus\mn_{3,2}, \quad
     \mn_{7,3,A}, \quad
     \mn_{7,3,B_1},\quad
     \mn_{7,3,C},\quad
\mh_\H.
\]
Then $M$ admits a co-closed $\mathrm{G}_2$-structure $\varphi$ and a connection $\hat{A}$ with gauge group $\mathrm{U}(1)\times \mathrm{SL}(2,\R)$ such that $(\varphi,\hat{A})$ is a solution to the heterotic $\G_2$-system \eqref{equ: heterotic system 1}-\eqref{equ: heterotic system 3}.
\end{cor}
\begin{remark}
    Note that the $\G_2$-structures in Corollary \ref{cor:delneq0} and \ref{cor:delneq0A4} are distinct, but nonetheless induce the same metric and orientation on $M$. Thus, our results demonstrate a curious phenomenon that there exist distinct isometric solutions to the heterotic $\G_2$-system (with different connections). This parallels the observation in \cite{finofow}*{Example 6.3} that there exists distinct strong $\G_2$-structures with torsion inducing the same metric.
\end{remark}

\subsection{\texorpdfstring{$3$}{}-Sasakian case}
\label{sec:3sasaki}

In this section we use our $\mathrm{SO}(3)$ ansatz to construct new solutions to the heterotic $\mathrm{G}_2$-system on certain $3$-Sasakian $7$-manifolds.
Recall that on a $3$-Sasakian manifold $M^7$ the dual $1$-forms to the Reeb vector fields satisfy:
\begin{equation}
    d e^5 = 2(e^{67}+\omega_1^+), \quad 
    d e^6 = 2(e^{75}+\omega_2^+), \quad
    d e^7 = 2(e^{56}+\omega_3^+),\label{equ: structure eq 3sasaki}
\end{equation}
where locally one can write
$\om_1^+=e^{12}+e^{34}$, $\om_2^+=e^{13}+e^{42}$ and $\om_3^+=e^{14}+e^{23}$ by choosing a local transverse orthonormal co-frame $\{e^i\}_{i=1}^4$.  
Thus, one can define a nearly parallel $\mathrm{G}_2$-structure on $M$ by
\begin{equation}
    \varphi_{ts}:=e^{567}+e^{5}\wedge \omega_1^++e^{6}\wedge\omega_2^+ -e^{7}\wedge \omega_3^+.\label{eq: 3sasa g2 structure}
\end{equation}
Indeed using the structure equations \eqref{equ: structure eq 3sasaki}, one easily verifies the nearly parallel condition:
\begin{equation*}
    d\varphi_{ts} = 4 \star_{ts} \varphi_{ts}. 
\end{equation*}
The $3$-Sasakian condition is equivalent to the fact that the cone metric $dr^2+r^2g_{ts}$ on $\R^+_r\times M^7$ has holonomy group contained in $\mathrm{Sp}(2)$  i.e. it is hyperK\"ahler  \cite{Galicki1996}*{Proposition 2.1}. Thus, we shall refer to $\varphi_{ts}$ as the $3$-Sasakian $\G_2$-structure (`ts' short for 'three-Sasakian').

It was shown in \cite{Galicki1996}*{Proposition 2.4} that $M^7$ admits another distinct nearly parallel $\mathrm{G}_2$-structure given by:
\begin{equation}
    \varphi_{np}:=-\frac{27}{125}e^{567}+\frac{27}{25}\big(e^{5}\wedge \omega_1^++e^{6}\wedge\omega_2^+ +e^{7}\wedge \omega_3^+\big).\label{eq: np g2 structure}
\end{equation}
The latter is obtained by suitably squashing the metric on the leaf of the foliation generated by $\langle e_5,e_6,e_7\rangle.$ Again using \eqref{equ: structure eq 3sasaki}, one can verify that it is nearly parallel:
\begin{equation*}
    d\varphi_{np} = 4 \star_{np} \varphi_{np}.
\end{equation*}
Unlike in the previous case, however, the associated cone metric $dr^2+r^2g_{np}$ has holonomy group \textit{equal} to $\mathrm{Spin}(7)$, see \cite{FriedrichNP}*{Theorem 5.5};  $\varphi_{np}$ is said to be \textit{strictly} nearly parallel. It is not hard to see that $\varphi_{ts}$ and $\varphi_{np}$ are neither isometric nor induce the same orientation on $M$ ({$\vol_{np}=-\frac{3^7}{5^2}\vol_{ts}$}); though they both define Einstein metrics with positive scalar curvature equal to $42$ (owing to our normalisation).

It follows automatically from the nearly parallel condition that \eqref{equ: heterotic system 1} holds for $\varphi_{ts}$ and $\varphi_{np}$. 
Next we consider when the connection $A$, given by \eqref{eq:A}, satisfies \eqref{equ: heterotic system 2}.
\begin{pro}\label{pro: instanton sasaki}
For the $3$-Sasakian $\mathrm{G}_2$-structure $\varphi_{ts}$ \eqref{eq: 3sasa g2 structure}, the connection $A$ given by \eqref{eq:A} is a $\mathrm{G}_2$-instanton if the gauge group is $\mathrm{SL}(2,\mathbb{R})$ with $\textbf{C} = \diag(
-6, -6,+2)$. 
On the other hand, for the strictly nearly parallel $\mathrm{G}_2$-structure $\varphi_{np}$ \eqref{eq: np g2 structure}, the connection $A$ given by \eqref{eq:A} is a $\mathrm{G}_2$-instanton if the gauge group is $\mathrm{SU}(2)$ with $\textbf{C} = -\frac{6}{5}\mathrm{Id}$. 
 \end{pro}
\begin{proof}
First we consider $\varphi_{ts}$. In this case, one checks easily that the $\mathrm{G}_2$ $4$-form is given by
\[
\psi_{ts}=\frac{1}{2}\omega_1^+ \w \omega_1^+ 
+ e^{67}\w \omega_1^+
- e^{57}\w \omega_2^+
- e^{56}\w \omega_3^+.
\]
Using the structure equations \eqref{equ: structure eq 3sasaki}, the instanton condition \eqref{eq: instanton condition 1} becomes:
 \begin{equation*}
    \big( (2+\lambda_5) e^{67} + 2\omega_1^+  \big)\wedge \psi_{ts} = 
    \big( -(2+\lambda_6) e^{57} + 2\omega_2^+  \big)\wedge \psi_{ts}= 
    \big( (2+\lambda_7) e^{56} + 2\omega_3^+  \big) \wedge \psi_{ts}=0.
\end{equation*}
Expanding the latter, we get the system:
 \begin{equation*}
   \frac{1}{2}(2+\lambda_5)+2=0,\qquad
   \frac{1}{2}(2+\lambda_{6} )+2=0,\qquad
   \frac{1}{2}(2+\lambda_{7 })-2=0,
\end{equation*}
yielding the solution $(\lambda_5,\lambda_6,\lambda_7)=(-6,-6,+2)$. The result in the $\varphi_{np}$ case follows by a similar computation using
\[
\psi_{np}=\frac{81}{50}\omega_1^+ \w \omega_1^+ -\frac{81}{125}\Big(
 e^{67}\w \omega_1^+
- e^{57}\w \omega_2^+
+ e^{56}\w \omega_3^+\Big),
\]
and one finds that $\lambda_5=\lambda_6=\lambda_7=-\frac{6}{5}$ in this case.
\end{proof}

\begin{remark}
It is worth emphasising that the $\mathrm{G}_2$-instantons arising from Proposition \ref{pro: instanton sasaki} are not obtained via pullback from lower dimensional constructions; this can easily be seen by inspection of the structure equations \eqref{equ: structure eq 3sasaki}.  
\end{remark}
Since the only non-zero $\mathrm{G}_2$ torsion form is $\tau_0=4$, from \eqref{equ: definition of T} we have
\begin{eqnarray}
    T_{ts}=\frac{2}{3}\varphi_{ts} \qquad \text{and} \qquad T_{np} = \frac{2}{3}\varphi_{np}.\label{eq: torsion 3form ts and np}
\end{eqnarray}
In order to find new solutions to the heterotic system \eqref{equ: heterotic system 1}-\eqref{equ: heterotic system 3}, we also need to choose an invariant inner product on the Lie algebra $\mathfrak{sl}(2,\R)$ and $\mathfrak{su}(2)$. We consider the non-degenerate pairing given by
\begin{equation}
\langle Y_i,Y_j \rangle=r^{-1}\delta_{ij}    \label{equ: r pairing}
\end{equation} 
on both of the Lie algebras of the gauge groups $G= \mathrm{SL}(2,\R)$ and $\mathrm{SU}(2)$, where $r\in \R\backslash\{0\}$ is a free parameter. For $\mathfrak{su}(2)$, this pairing corresponds to the Killing form (up to a constant) so we shall simply write $\langle \cdot,\cdot \rangle_{\mathfrak{su}(2)}$. On the other hand, for $\mathfrak{sl}(2,\R)$ the Killing form has signature $(2,1)$ hence in this case this pairing is  not $\mathrm{ad}$-invariant, so we shall denote it by $\langle \cdot,\cdot \rangle_{\mathfrak{sl}(2,\R)'}$. 

Denoting by $A_{ts}$ and $A_{np}$ the $\mathrm{G}_2$-instantons obtained from Proposition \ref{pro: instanton sasaki}, a direct computation shows that
\begin{equation*}
    r\langle F_{A_{ts}}\w F_{A_{ts}} \rangle_{\mathfrak{sl}(2,\R)'} = -16 \psi_{ts}+20 (\omega_1^+ \w \omega_1^+)
\quad\text{and}\quad
    r\langle F_{A_{np}}\w F_{A_{np}} \rangle_{\mathfrak{su}(2)} = -\frac{400}{81}\psi_{np}+20 (\omega_1^+ \w \omega_1^+).
\end{equation*}
From \eqref{eq: torsion 3form ts and np} we have:
\begin{eqnarray*}
    dT_{ts}=\frac{8}{3}\psi_{ts} \qquad \text{and}\qquad dT_{np} = \frac{8}{3}\psi_{np},
\end{eqnarray*}
and thus, we can rewrite the above equivalently as:
\begin{equation}
     dT_{ts}=-\frac{rt}{6}\langle F_{A_{ts}}\w F_{A_{ts}} \rangle_{\mathfrak{sl}(2,\R)'}  +\frac{20t}{6} (\omega_1^+ \w \omega_1^+)
     +\frac{8}{3}(1-t)\psi_{ts}\label{equ: dTts}
\end{equation}
and
\begin{equation}
    dT_{np}=-\frac{27rt}{50}\langle F_{A_{np}} \w F_{A_{np}} \rangle_{\mathfrak{su}(2)} +\frac{54t}{5} (\omega_1^+ \w \omega_1^+)+\frac{8}{3}(1-t)\psi_{np},\label{equ: dTnp}
\end{equation}
where $t\in \R$ is a free parameter. Before describing new examples, we recall the solution to the heterotic $\mathrm{G}_2$-system found in \cite{Ivanov2005} which can viewed as the special case when $t=0$ in \eqref{equ: dTts}. 
\begin{ex}\label{example: ivanov-ivanov}
In \cite{Ivanov2005}*{\S 6}, see also \cite{Lotay2024}*{Example 4.17}, Ivanov-Ivanov showed that for $M=S^7$ endowed with $\mathrm{G}_2$-structure $\varphi_{ts}$, the associated characteristic connection $\nabla^c$ satisfies the heterotic Bianchi identity:
    \[
    \langle F_{\nabla^c}\w F_{\nabla^c}\rangle_{\mathfrak{g}_2}= -\frac{32}{27}\psi_{ts} = -\frac{4}{9}\d T_{ts},
    \]
    where $\langle \cdot,\cdot\rangle_{\mathfrak{g}_2}$ corresponds to the standard $\mathrm{ad}$-invariant pairing on $\mathfrak{g}_2$. Furthermore, from \cite{Harland2011}*{Corollary 3.2}, for any nearly parallel $\mathrm{G}_2$-structure, the associated characteristic connection $\nabla^c$ is always a $\mathrm{G}_2$-instanton. Thus, this provides a solution to the heterotic $\mathrm{G}_2$-system (after suitably scaling the $\mathrm{ad}$-invariant form). Note that this example is only known for $S^7$ endowed with $\varphi_{ts}$; indeed the calculation in \cite{Ivanov2005} uses that fact the induced round metric $g_{ts}$ has constant curvature.
    For general nearly parallel $\mathrm{G}_2$-structures, it is rather tedious to compute $\langle F_{\nabla^c}\w F_{\nabla^c}\rangle_{\mathfrak{g}_2}$, and no general computation is known to us. 
\end{ex}
We now describe new solutions arising from our ansatz.
\begin{ex}\label{example: solution on S7}
Consider again $M=S^7$. It is well-known that $S^4$ admits an anti-self-dual instanton $A_{ASD}$ with gauge group $\mathrm{SU}(2)$ induced by its Levi-Civita connection on $\Lm^2_-(S^4)$ \cite{AtiyahASD}. Pulling back this connection via the Hopf fibration $S^3 \hookrightarrow S^7\to S^4$, this gives a $\mathrm{G}_2$-instanton for both $\varphi_{ts}$ and $\varphi_{np}$ (under the inclusion $\mathfrak{su}(2)\subset\mathfrak{g}_2$). Furthermore, one has
\[
\langle F_{A_{ASD}}\w F_{A_{ASD}}\rangle_{\mathfrak{su}(2)}= -\omega_1^+ \w \omega^+_1,
\]
where $\langle \cdot, \cdot \rangle_{\mathfrak{su}(2)}$ corresponds to the Killing form (scaled by a suitable constant factor). Together with the above example of Ivanov-Ivanov, we can rewrite \eqref{equ: dTts} as:
\begin{equation*}
     dT_{ts}=-\frac{rt}{6}\langle F_{A_{ts}}\w F_{A_{ts}} \rangle_{\mathfrak{sl}(2,\R)'}  -\frac{20t}{6} \langle F_{A_{ASD}}\w F_{A_{ASD}}\rangle_{\mathfrak{su}(2)}
     -\frac{9}{4}(1-t)\langle F_{\nabla^c}\w F_{\nabla^c}\rangle_{\mathfrak{g}_2}.
\end{equation*}
Hence, we have the following: 
\begin{itemize}
    \item If we consider the product connection $A:=A_{ts}\oplus A_{ASD}\oplus \nabla^c$ on the principal $\mathrm{SL}(2,\R) \times \mathrm{SU}(2)\times \mathrm{G}_2$-bundle, and scale the parings $\langle\cdot,\cdot \rangle$ appropriately, we get a family of new solutions (depending on $t$) to \eqref{equ: heterotic system 1}-\eqref{equ: heterotic system 3} on $(S^7,\varphi_{ts})$.
    \item For $t=1$, we can  also consider the connection $A:=A_{ts}\oplus A_{ASD}$ on the principal $\mathrm{SL}(2,\R) \times \mathrm{SU}(2)$-bundle over $(S^7,\varphi_{ts})$ and $t=0$ recovers Example \ref{example: ivanov-ivanov}.
\end{itemize}
We emphasise here that the pairing \eqref{equ: r pairing} on $\mathrm{SL}(2,\R)$ is only left invariant but not bi-invariant (compare instead with Example \ref{example: ts twisted} below). 

For $(S^7,\varphi_{np})$, it is not known if $\langle F_{\nabla^c}\w F_{\nabla^c}\rangle_{\mathfrak{g}_2}$ is proportional to $\psi_{np}$. In this case, we have the following:
\begin{itemize}
    \item Setting $t=1$ in \eqref{equ: dTnp}, we can choose a $\mathrm{ad}$-invariant pairing on $\mathfrak{su}(2)\oplus \mathfrak{su}(2)$, so that the connection $A:=A_{np}\oplus A_{ASD}$ on a principal $\mathrm{SU}(2) \times \mathrm{SU}(2)$-bundle over $(S^7,\varphi_{np})$ solves the heterotic $\mathrm{G}_2$-system \eqref{equ: heterotic system 1}-\eqref{equ: heterotic system 3}. 
\end{itemize}
\end{ex}

\begin{ex}\label{example: solution on N11}
    Consider now the Aloff-Wallach space $M=N^{1,1}:=\mathrm{SU}(3)/\mathrm{U}(1)_{1,1}$, where $\mathrm{U}(1)_{1,1}=\diag(e^{i\theta},e^{i\theta},e^{-2i\theta})$. We proceed along the same lines as in the previous example. In this case, we can take the pullback of the Fubini-Study form $\omega_{FS}$ on $\mathbb{CP}^2$ via the fibration: 
    $$\mathrm{SO}(3) \hookrightarrow N^{1,1} \to \mathbb{CP}^2= \mathrm{SU}(3)/\mathrm{U}(2).$$
    Here we are viewing $\omega_{FS}$ as the curvature $2$-form of a connection $1$-form $\alpha_{FS}$ on the Hopf bundle $S^1\hookrightarrow S^5 \to \mathbb{CP}^2$. It is not hard to see that $\omega_{FS}$ is also a $\mathrm{G}_2$-instanton for both $\varphi_{ts}$ and $\varphi_{np}$; this follows from the fact that $\omega_{FS}\w \omega_i^+=0$ for $i=1,2,3$, hence we can assume $\om_{FS}\w \omega_{FS}=-\omega_1^+\w \omega_1^+$. 
    \begin{itemize}
        \item Setting $t=1$ in \eqref{equ: dTts} and \eqref{equ: dTnp}, after suitable normalisation of the pairing \eqref{equ: r pairing} we have the solutions:
    $A=A_{ts}\oplus k\alpha_{FS}$ on the principal $\mathrm{SL}(2,\R)\times \mathrm{U}(1)$ bundle over $(N^{1,1},\varphi_{ts})$ and $A=A_{np}\oplus k\alpha_{FS}$ on the principal $\mathrm{SU}(2)\times \mathrm{U}(1)$ bundle over $(N^{1,1},\varphi_{np})$, where $k\in \mathbb{Z}\backslash\{0\}$, to \eqref{equ: heterotic system 1}-\eqref{equ: heterotic system 3}.
    \end{itemize}
    The above solutions are analogous to those in Example \ref{example: solution on S7} whereby $A_{ASD}$ is replaced by $\alpha_{FS}$.
\end{ex}

We now use our $\mathrm{SO}(3)$ ansatz to construct more solutions on $3$-Sasakian $7$-manifolds. Consider the $\mathrm{G}_2$-structure given by
\begin{equation}
    \widehat{\varphi}_{ts}:=
    e^{567}
    -e^{5}\wedge \omega_1^+
    -e^{6}\wedge\omega_2^+ 
    -e^{7}\wedge \omega_3^+.\label{eq: 3sasa g2 structure modified}
\end{equation}
In terms of the $\mathrm{SO}(3)$ ansatz \eqref{eq:phians},
$\widehat{\varphi}_{ts}$ corresponds to setting $\textbf{B}=\diag( -1, -1,+1)$, whereby $\textbf{B}=\mathrm{Id}$ corresponds to $\varphi_{ts}$ \eqref{eq: 3sasa g2 structure}. In particular, the underlying metric and orientation are both unchanged. However, in contrast to $\varphi_{ts}$, the $\mathrm{G}_2$-structure determined by $\widehat{\varphi}_{ts}$ is not nearly parallel. It is still co-closed hence satisfies \eqref{equ: heterotic system 1}, and a direct computation shows:
\begin{equation*}
    d\widehat{\varphi}_{ts} = -\frac{36}{7}\star_{ts}\widehat{\varphi}_{ts} + \star_{ts}\Big(-\frac{48}{7}e^{567}-\frac{8}{7}(e^5\w \omega_1^++e^6\w \omega_2^++e^7\w \omega_3^+)\Big).
\end{equation*}
Comparing with \eqref{equ: definition of T} we deduce that
\begin{equation*}
 T_{\widehat{\varphi}_{ts}} =  6 e^{567} + 2(e^5\w \omega_1^++e^6\w \omega_2^++e^7\w \omega_3^+).
\end{equation*}
Next we consider the instanton condition:
\begin{pro}\label{prop: modified instanton}
    For the $\mathrm{G}_2$-structure $\widehat{\varphi}_{ts}$ \eqref{eq: 3sasa g2 structure modified}, the connection $A$ given by \eqref{eq:A} is a $\mathrm{G}_2$-instanton if the gauge group is $\mathrm{SU}(2)$ with $\textbf{C} = 2 \mathrm{Id}$. 
\end{pro}
\begin{proof}
This follows by an analogous computation as in Proposition \ref{pro: instanton sasaki}.
\end{proof}
Comparing with Proposition \ref{pro: instanton sasaki}, observe that while ${\varphi}_{ts}$ and $\widehat{\varphi}_{ts}$ are both isometric and induce the same orientation, the connection $A$ given by \eqref{eq:A} is a $\mathrm{G}_2$-instanton with different gauge group in each case.

\begin{ex}\label{example: ts twisted}
Denoting the $\mathrm{G}_2$-instanton from Proposition \ref{prop: modified instanton} by $\widehat{A}_{ts}$, using the structure equations \eqref{equ: structure eq 3sasaki} a long but straightforward calculation shows:
\begin{equation*}
d T_{\widehat{\varphi}_{ts}} = \langle F_{\widehat{A}_{ts}}\w F_{\widehat{A}_{ts}}\rangle_{\mathfrak{su}(2)} + 6 \omega_1^+ \w \omega_1^+,
\end{equation*}
where $\langle\cdot, \cdot\rangle_{\mathfrak{su}(2)} $ corresponds to the $\mathrm{SU}(2)$ Killing form normalised so that $\langle Y_i, Y_j\rangle_{\mathfrak{su}(2)}=\frac{1}{2}\delta_{ij}$. We can now do the same trick as in Example \ref{example: solution on S7} and \ref{example: solution on N11}, yielding: 
\begin{itemize}
    \item $A:=\hat{A}_{ts}\oplus A_{ASD}$ on $(S^7,\widehat{\varphi}_{ts})$ with gauge group $\mathrm{SU}(2)\times \mathrm{SU}(2)$, 
    \item $A:=\hat{A}_{ts}\oplus k \alpha_{FS}$ on $(N^{1,1},\widehat{\varphi}_{ts})$ with gauge group $\mathrm{SU}(2)\times \mathrm{U}(1)$,
\end{itemize}
both solving \eqref{equ: heterotic system 1}-\eqref{equ: heterotic system 3}. In contrast to the solutions for the nearly parallel $\mathrm{G}_2$-structure $\varphi_{ts}$ in Example \ref{example: solution on S7} and \ref{example: solution on N11}, the gauge group is now compact and the pairing on the Lie algebra corresponds to a $\mathrm{ad}$-invariant one.
\end{ex}

\begin{remark}
Similar to the above example, one can consider the $\mathrm{G}_2$-structure defined by the $3$-form $\widehat{\varphi}_{np}$, which is isometric to $\varphi_{np}$ and with same orientation such that $\widehat{\varphi}_{np}$ corresponds to $\textbf{B}=\diag(-1,-1,+1)$ and $\varphi_{np}$ to $\textbf{B}=\mathrm{Id}$. In this case, one finds that $A$, given by \eqref{eq:A}, is a $\mathrm{G}_2$-instanton with respect to $\widehat{\varphi}_{np}$ if the gauge group is again $\mathrm{SU}(2)$ but now with $\textbf{C}=\diag(\frac{14}{5},\frac{14}{5},\frac{6}{5})$. Unfortunately, in this case we have not been able to find a solution to the heterotic Bianchi identity \eqref{equ: heterotic system 3}.
\end{remark}

\section{
\texorpdfstring{$S^1$}{}-family of integrable 
\texorpdfstring{$\rmG_2$}{}-structures}\label{sec: s1 invariant}

In this section, we consider integrable $\mathrm{G}_2$-structures arising on an $S^1$-bundle over a $6$-manifold endowed with an $\SU(3)$-structure. We show that, under certain torsion conditions of the latter, the total torsion of the former is constant for all values of the parameter and they provide solutions to the heterotic $\G_2$-system. We illustrate a few applications in explicit examples.
The results in this section extend those in \cite{FinoG2T2023}*{\S 4} to the case when the $S^1$-bundle is not necessarily a product and with non-trivial connection.

Let $(Q, \omega,\Upsilon_+)$ denote a $6$-manifold endowed with an $\SU(3)$-structure and let $M$ be a principal $S^1$-bundle over $Q$ endowed with a connection $1$-form $\eta$. We can then define a natural one-parameter family of $S^1$-invariant $\rG_2$-structures on $M$ by
\begin{equation}\label{eq: S1-varphi}
\begin{gathered}
\varphi_t = \eta \wedge \omega + \mathrm{Re}\big(e^{it}(\Upsilon_++i\Upsilon_-)\big), \\
    \psi_t = \frac{1}{2}\omega \wedge \omega -\eta \wedge \mathrm{Im}\big(e^{it}(\Upsilon_++i\Upsilon_-)\big),
\end{gathered}
\end{equation}
where $t\in[0,2\pi)$. The curvature $2$-form $\d\eta$ descends to $Q$ and defines an integral cohomology class in $H^2(Q,\mathbb{Z})$. Throughout this section we shall identify tensors on $Q$ with their pullbacks to $M$.

We emphasise that while each $\varphi_t$ defines a distinct $\G_2$-structure for different $t$, they all nonetheless induce the same metric:
\[
g_{\varphi} = \eta \otimes \eta + g_{\omega},
\]
and orientation $\vol_\varphi=\eta \w \vol_{\omega}$ on $M$. Here $g_\omega$ is the metric in $Q$ induced by the $\SU(3)$-structure (see Section \ref{sec:backsu3} for notation). In order to distinguish between the Hodge star operators associated to $g_\varphi$ and $g_\omega$, we shall denote them by $\star_7$ and $\star_6$, respectively. In particular, for any $k$-form $\alpha$ on $Q$, we have
\begin{equation}\label{eq: star7_star6}
    \star_7\alpha=(\star_6\alpha)\wedge\eta, \qquad \star_7(\alpha\wedge\eta)=(-1)^k\star_6\alpha.
\end{equation}
We can encode the integrable condition for $\varphi_t$ in \eqref{eq: S1-varphi} in terms of data on $(Q, \omega,\Upsilon_+)$ as follows:

\begin{pro}\label{prop: S1- solutions}
The $\rG_2$-structure defined by \eqref{eq: S1-varphi} is integrable, i.e. $\tau_2=0$, for all $t\in [0,2\pi)$ if and only if $\sigma_2=\pi_2=0$, $\pi_1=2\nu_1$ and $\d\eta$ is $J$-invariant i.e. of type $(1,1)$. In this case, the torsion $3$-form $T_{\varphi_t}$ is explicitly given by
\begin{align}\label{eq:Tcte general}
    \begin{split}
        T_{\varphi_t} 
        &= \eta \wedge \Big(\d\eta -2{(\d\eta)_0}\omega\Big) + T_\omega+{(d\eta)_0} \mathrm{Re}(e^{it}{\Upsilon}), 
        \end{split}
    \end{align}
where $(d\eta)_0$ denotes the $\omega$-component of $d\eta$ and $T_\omega$ is the torsion form of the Bismut connection of $(\omega,\Upsilon_+)$ given by \eqref{eq: def T_omega}.
In particular, if
$d\eta$ is traceless then $T_{\vp_t}$ is independent of $t$, and it is simply given by
\begin{align}\label{eq:Tcte}
    \begin{split}
        T_{\varphi_t} 
        &= \eta \wedge \d\eta + T_\omega. 
        \end{split}
    \end{align} 
\end{pro}
\begin{proof}
Let us write $\tilde{\Upsilon}_+ +i\tilde{\Upsilon}_-:=e^{it}(\Upsilon_++i\Upsilon_-)$ for the complex $(3,0)$-form. The torsion forms \eqref{eq:torsu3} of the $\mathrm{SU}(3)$-structure determined by $(\omega, \tilde{\Upsilon}_+)$ are given by
\begin{equation*}
\d(\tilde{\Upsilon}_+ +i\tilde{\Upsilon}_-) = e^{it}(\pi_0+i\sigma_0) \omega^2 + \pi_1 \w (\tilde{\Upsilon}_+ +i\tilde{\Upsilon}_-) - e^{it}(\pi_2 + i\sigma_2) \wedge \omega.
\end{equation*}
We see immediately that $\tilde{\pi}_1=\pi_1$ and
\begin{gather}
\begin{split}\label{equ: cond2}
    \tilde{\pi}_0 = \mathrm{Re}(e^{it}(\pi_0+i\sigma_0)),\qquad 
    \tilde{\sigma}_0 = \mathrm{Im}(e^{it}(\pi_0+i\sigma_0)),\\
    \tilde{\pi}_2 = \mathrm{Re}(e^{it}(\pi_2+i\sigma_2)),\qquad 
    \tilde{\sigma}_2 = \mathrm{Im}(e^{it}(\pi_2+i\sigma_2)).
\end{split}
\end{gather}
From \cite{finofow}*{Theorem 5.5}, it follows that the $S^1$-invariant  $\mathrm{G}_2$-structure $\varphi_t$ is integrable if and only if $\tilde{\sigma}_2=0$ and 
\begin{equation}
    2(\d\eta)^2_6 \w \tilde{\Upsilon}_- - 2 \nu_1 \w \om^2 + \pi_1 \w \om^2 = 0.\label{eq: cond1}
\end{equation}
Furthermore, in this case the torsion $3$-form is given by
\begin{align}
    T_{\varphi_t} =\ & \eta \w (\star_6 (\frac{2}{3}\nu_1-\frac{1}{3}\pi_1) \w \tilde{\Upsilon}_+ - (\d\eta)_0 \om-\frac{1}{3}(\d\eta)^2_6+(\d\eta)^2_8-\tilde{\pi}_2)\label{equ: cond3}\\
    &\hspace{-3mm}+\Big((\d\eta)_0\tilde{\Upsilon}_+ +\frac{1}{2}(\tilde{\sigma}_0 \tilde{\Upsilon}_- +\tilde{\pi}_0\tilde{\Upsilon}_+) + \star_{6}\big( \nu_3-(\frac{2}{3}\pi_1-\frac{1}{3}\nu_1-\frac{1}{3}\star_{6}(\d\eta \w \tilde{\Upsilon}_+))\w \om \big) \Big)\nonumber,
\end{align}
where $(\d \eta)_0$ denotes the $\omega$-component of $\d\eta$ with respect to \eqref{eq:L2su3}, see \cite{finofow}*{(50)}. It is not hard to see that \eqref{eq: cond1} holds for all $t\in [0,2\pi)$ precisely if $(d\eta)^2_6=0$, i.e. $d\eta \in \langle\omega\rangle\oplus \Lm^2_8$, and $\pi_1=2\nu_1$. Secondly, from \eqref{equ: cond2} we see that $\tilde{\sigma}_2=0$ for all $t$ precisely if $\pi_2=\sigma_2=0$. This proves the first assertion.

Now we want to find the condition so that $T_{\varphi_t}=T_{\varphi_0}$ for all $t\in [0,2\pi)$. By inspection of \eqref{equ: cond3} we see that we need $(d\eta)_0\tilde{\Upsilon}_+$ to vanish; since $(\d\eta)^2_6=0$, this is equivalent to requiring $d\eta \in \Lm^2_8$. Finally, a direct computation shows that for any $t$:
\[
\tilde{\sigma}_0 \tilde{\Upsilon}_- + \tilde{\pi}_0 \tilde{\Upsilon}_+ = {\sigma}_0 {\Upsilon}_- + {\pi}_0 {\Upsilon}_+,
\]
so $T_{\varphi_t}$ is independent of $t$ and this concludes the proof.
\end{proof}

Henceforth, we shall assume that the $\rG_2$-structures defined by \eqref{eq: S1-varphi} are all integrable. Note that the intrinsic $\rG_2$ torsion forms $\tau_i$ do depend on $t\in [0,2\pi)$ in general; even if  $T_{\varphi_t}$ is $t$-independent. More concretely, we have:
\begin{cor}\label{cor:tforms}
    Let $\varphi_t$ be the integrable $\rG_2$-structure defined by \eqref{eq: S1-varphi} with torsion \eqref{eq:Tcte general}. Then, its intrinsic torsion forms $\tau_0$ and $\tau_1$ are given by:
    \begin{equation*}
        \tau_0(t)=\frac{12}{7}\left(\cos(t)\pi_0-\sin(t)\sigma_0\right)+{\frac{6}{7}(\d\eta)_0}, \qquad \tau_1(t)=\frac12\left(\cos(t)\sigma_0-\sin(t)\pi_0\right)\eta+\frac{1}{2}\nu_1.
    \end{equation*}
\end{cor}

\begin{proof}
    The result follows from the expression \eqref{eq: S1-varphi} and \eqref{eq:Tcte} into the formulae:
    $$
    \tau_0(t)=\frac{6}{7}g_{\varphi}( T_{\varphi_t},\varphi_t) \qandq \tau_1(t)=-\frac14\star_7(T_{\varphi_t}\wedge\varphi_t). \vspace{-8mm}$$
\end{proof}

As an application of Proposition \ref{prop: S1- solutions}, we show how one can lift solutions to the heterotic $\mathrm{SU}(3)$-system to $S^1$-invariant solutions to the heterotic $\mathrm{G}_2$-system. 

\begin{teo}\label{thm: S1- solutions}
     Let $M$ be a principal $S^1$-bundle over  $(Q,\omega,\Upsilon_+)$ with connection form $\eta$ such that $d\eta \in \Lambda^2_8$ and the torsion forms of $(\omega,\Upsilon_+)$ satisfy $\sigma_2=\pi_2=0$ and $\pi_1=2\nu_1$. Assume that $(\omega,\Upsilon_+,A)$ solves the $\SU(3)$ heterotic Bianchi identity \eqref{eq: het_SU3-Bianchi-ident}, where $A$ is an $\SU(3)$-instanton with gauge group $G$. Then $(\varphi_t,\eta\oplus A)$ with $\varphi_t$ given in \eqref{eq: S1-varphi} is a solution of the heterotic $\rG_2$-system \eqref{equ: heterotic system 1}-\eqref{equ: heterotic system 3} for any $t\in[0,2\pi)$.
\end{teo}
\begin{proof}
    By Proposition \ref{prop: S1- solutions}, we have $\varphi_t$ integrable with $T_{\varphi_t}=\eta\wedge \d\eta+T_\omega$ independent of $t$ since $d\eta\in \Lambda^2_8$. If $(\omega,\Upsilon_+, A)$ solves \eqref{eq: het_SU3-Bianchi-ident} then 
    $$    F_{A}\wedge\psi_t=0 \qquad\text{and} \qquad \d T_{\varphi_t}=\langle F_{\eta\oplus A}\wedge F_{\eta\oplus A}\rangle,
    $$
    where $\langle\cdot,\cdot\rangle:=\langle\cdot,\cdot\rangle_{\mathfrak{u}(1)}\oplus\langle\cdot,\cdot\rangle_{\mathfrak{g}}$.
\end{proof}

We next give a few applications of the above results with some explicit examples; several of which appear to be new.

\begin{ex}[Solutions with $dT_\omega=0$] \label{ex: on s3s3}
Consider $Q=S^3\times S^3$ with the
usual left invariant co-framing $\{e_i\}^6_{i=1}$ satisfying:
\begin{equation}\label{eq:  coframe S3xS3}
de^1 = -2 e^{23}, \quad 
de^2 = -2 e^{31}, \quad 
de^3 = -2 e^{12}, \quad 
de^4 = -2 e^{56}, \quad 
de^5 = -2 e^{64}, \quad 
de^6 = -2 e^{45}. 
\end{equation}
We define an $\SU(3)$-structure on $Q$ by 
\begin{equation}
\begin{gathered}\label{eq: SU3_Udhav_example}
\omega = f^{12}+f^{34}+f^{56},\\
\Upsilon = (f^1+if^2)\wedge(f^3+if^4)\wedge(f^5+if^6),
\end{gathered}
\end{equation}
where $f^{2i-1} := \frac12(e^i - e^{i+3})$ and  $f^{2i} := \frac12(e^i + e^{i+3})$ for $i=1,2,3$.
It was shown in \cite{finofow}*{Example 5.8} that the only non-zero torsion forms of the $\SU(3)$-structure given by \eqref{eq: SU3_Udhav_example} are $\sigma_0$ and $\nu_3$, and that $T_\omega$, given by \eqref{eq: def T_omega}, is closed. Consider $M=S^3\times S^3\times S^1$ with the connection $1$-form $\eta$  satisfying $d\eta=f^{13}+f^{24}=-\frac{1}{4}d(e^3+e^6)$. It is not hard to verify that $d\eta \in \Lm^2_8(Q)$. Hence, Theorem \ref{thm: S1- solutions} yields a $1$-parameter family of $\G_2$-structures $\varphi_t$, given by \eqref{eq: S1-varphi}, solving \eqref{equ: heterotic system 1} -\eqref{equ: heterotic system 3} with $A=\eta$. Furthermore, from Corollary \ref{cor:tforms}, we see that one can have $\tau_0=0$ and $\tau_1\neq 0$, $\tau_0\neq 0$ and $\tau_1= 0$, or both $\tau_0$ and $\tau_1$ non-zero.
\end{ex}

\begin{ex}[Solutions with $dT_\omega \neq 0$: revisiting $\R \op \mathfrak{n}_{3,2}$] \label{ex: n32 again} Recall from Corollary \ref{cor:delneq0} that $M=S^1\times (\Gamma\backslash N_{3,2})$ admits a solution $(\varphi,A)$ to the heterotic $\G_2$-system with $G=\SU(2)$. Moreover from Remark \ref{rem: relation to characteristic connection}, $A$ can also be identified with the characteristic connection $\nabla^c$ of $\varphi$. We describe this solution explicitly as follows. Let
$\mn_{3,2}={\rm span}\{e_2, \ldots, e_7\}$ with structure equations: $de^i=0$, where $i=2,3,4$ and 
\begin{equation*}
    de^5= -2 e^{24},\qquad  de^6= -2 e^{23},\qquad  de^7= 2 e^{34}.
\end{equation*}
Denoting by $\eta=e^1$ the co-framing on the $S^1$ factor, the $\G_2$-structure $\varphi$ is given by \eqref{eq:phians}. The $\SU(3)$-structure induced on $N_{3,2}$ is given by
\begin{equation}\label{eq: su3 n32}
 \omega =    e^{27}+e^{35}-e^{46}, \quad \Upsilon_+=e^{347}+e^{567}-e^{236}-e^{245}, \quad \Upsilon_-=-(e^{234}+e^{256}+e^{457}+e^{367}).
\end{equation}
In our previous notation, $\textbf{A}^+=\textbf{B}=\mathrm{Id}$ and $\textbf{A}^-=\diag(+1,-1,+1)$.
A direct calculation shows that the only non-zero $\SU(3)$ torsion forms for \eqref{eq: su3 n32} are $\pi_0=1$ and $\nu_3$, and that 
\[
F_A=-2(e^{24}+e^{67})\otimes Y_1
-2(e^{23}-e^{57})\otimes Y_2
+2(e^{34}-e^{56})\otimes Y_3.
\]
We can apply the rotational ansatz \eqref{eq: S1-varphi} to get a $1$-parameter family of $\G_2$-structures $\varphi_t$ (these do \textit{not} occur in the ansatz \eqref{eq:phians} using $\textbf{B}$). By Proposition \ref{prop: S1- solutions} and Corollary \ref{cor:tforms},  $\vp_t$ is integrable and the torsion forms $\tau_0$ and $\tau_1$ are given by:
\begin{equation}\label{eq:n32tforms}
\tau_0=\frac{12}7\cos(t), \qquad \tau_1=-\frac12\sin(t)\eta.
\end{equation}
In particular, this shows that the $\G_2$-structures in the family are not equivalent for different values of the parameter $t$. The torsion  $3$-form $T_{\varphi_t}$ is, however, $t$-invariant and is explicitly given by
\[T_{\varphi_t}=T_\omega=-2e^{236}- 2e^{245}+2e^{347}-4e^{567}.\]
It is easy to see that $F_A \wedge \psi_t=0$ for all $t\in [0,2\pi)$. Thus, this gives a $1$-parameter family of solution to the heterotic $\G_2$-system \eqref{equ: heterotic system 1}-\eqref{equ: heterotic system 3} on $M=S^1\times (\Gamma\backslash N_{3,2})$ with same $A$ but varying $\tau_0$ and $\tau_1$ (in particular, these are not co-closed in general). 
\end{ex}

\begin{remark}
The one parameter family $\varphi_t$ in Example \ref{ex: n32 again} contains both the co-closed $\rG_2$-structure in \cite{CdBM}*{Example 5.3} for $t=0$, and the integrable but not co-closed structure in \cite{Ivanov2005}*{\S 6.2} for $t=\pi/2$ (see Eq.~\eqref{eq:n32tforms}).
\end{remark}

\begin{ex}[Solutions with $(d\eta)_0\neq0$]\label{ex: more characteristic connection}
Consider $S^1\hookrightarrow M^7\to \mathbb{T}^6$ with connection $1$-form $\eta$ satisfying 
\begin{equation*}
    d \eta = 
    a e^{12}+ 
    b e^{34}+ 
    c e^{56},
\end{equation*}  
where $a,b,c\in \R$ and $\{e^1,\ldots, e^6\}$ denote the standard flat $\mathrm{SU}(3)$ co-framing on $\mathbb{T}^6$. 
The manifold $M$ can be viewed as a nilmanifold whose nilpotent Lie algebra $\mathfrak{n}$ is isomorphic to either one of the following:
\begin{equation*}
    \R^7, \qquad 
    \R^4 \oplus \mathfrak{h}_3, \qquad 
    \R^2 \oplus \mathfrak{h}_5, \qquad 
    \mathfrak{h}_7,
\end{equation*}
depending on the parameters $a,b,c$. It is easy to see that $d\eta$ is of type $(1,1)$ and $(d\eta)_0=\frac{1}{3}(a+b+c)$.  Let $\varphi_t$ be the $S^1$-invariant family of co-closed $\G_2$-structures on $M$ given in \eqref{eq: S1-varphi}. If $(d\eta)_0=0$ i.e. $d\eta\in \Lambda^2_8$, then from Proposition \ref{prop: S1- solutions} we have
\[
dT_{\varphi_t } =d\eta \wedge d\eta.
\]
By Theorem \ref{thm: S1- solutions}, the latter yields a solution to the heterotic $\mathrm{G}_2$-system with $A=\eta$ i.e.~this is an abelian connection with gauge group $\mathrm{U}(1)$. In fact, one can show that $A$ corresponds to the characteristic connection of $(M^7,{ \varphi_t})$. More precisely, there is an embedding $A\in \Omega^1(\mathfrak{u}(1))\hookrightarrow\Omega^1(\mathfrak{g}_2)$ which corresponds to the characteristic connection $\nabla^c$ on $TM$, put differently $\mathfrak{hol}(\nabla^c)\cong \R$; this follows from the computations in  \cite{Fernandez2011}*{\S 5.1}, compare also with Remark \ref{rem: relation to characteristic connection} above. In this case, the condition $(d\eta)_0=0$ implies that $\mathfrak{n}\cong \R^2 \oplus \mathfrak{h}_5$ or $\mathfrak{h}_7$.

In the general case when $(d\eta)_0\neq 0$, there still exist solutions to \eqref{equ: heterotic system 1}-\eqref{equ: heterotic system 3}. To see this, we first define:
\[
\hat{\sigma}_1=e^{12}-e^{34},\qquad
\hat{\sigma}_2=e^{34}-e^{56},\qquad
\hat{\sigma}_3=e^{15}+e^{26}.
\]
It is easy to verify that $\hat{\sigma}_i\w \psi_t=0$. Furthermore, a long but straightforward computation shows: 
\[
dT_{\varphi_t} = \frac{1}{3}\Big(
(a^2+b^2-ab+ac+bc)(\hat{\sigma}_1 \w \hat{\sigma}_1)+
(b^2+c^2+ab+ac-bc)(\hat{\sigma}_2 \w \hat{\sigma}_2)+
(a^2+c^2+ab-ac+bc)(\hat{\sigma}_3 \w \hat{\sigma}_3)
\Big).
\]
Thus, let $\hat\xi_i$ denote connection $1$-forms with curvature $\d\hat\xi_i = \hat{\sigma}_i$ on $M$ then the abelian connection $A=(\hat\xi_1,\hat\xi_2,\hat\xi_3)$ with gauge group $\mathrm{U}(1)^3$ solve the heterotic Bianchi identity \eqref{equ: heterotic system 3} after a suitable choice of an $\mathrm{ad}$-invariant pairing on $\mathfrak{u}(1)^3\cong \R^3$ 
(for instance, by choosing a diagonal pairing with coefficients $(a-b)^2+d$, $(b-c)^2+d$ and $(a-c)^2+d$ with $d=ab+bc+ac$). This generalises the aforementioned examples in \cite{Fernandez2011}.
\end{ex}



\bibliography{biblio}

\begin{bibdiv}
\begin{biblist}

\bib{AtiyahASD}{article}{
      author={Atiyah, M.~F.},
      author={Hitchin, N.~J.},
      author={Singer, I.~M.},
       title={{Self-Duality in Four-Dimensional Riemannian Geometry}},
        date={1978},
     journal={Proc. Roy. Soc. Lond. A},
      volume={362},
       pages={425\ndash 461},
}

\bib{BFF18}{article}{
      author={Bagaglini, L.},
      author={Fern{\'a}ndez, M.},
      author={Fino, A.},
       title={Coclosed {${\mathrm {G}}_2$}-structures inducing nilsolitons},
        date={2018},
     journal={Forum Math.},
      volume={30},
      number={1},
       pages={109\ndash 128},
}

\bib{BallGoncalo}{article}{
      author={Ball, G.},
      author={Oliveira, G.},
       title={Gauge theory on {Aloff--Wallach} spaces},
        date={2019},
     journal={{Geom. Topol.}},
      volume={23},
      number={2},
       pages={685\ndash 743},
}

\bib{BeVe07}{article}{
      author={Bedulli, L.},
      author={Vezzoni, L.},
       title={The {Ricci} tensor of {SU}(3)-manifolds},
        date={2007},
        ISSN={0393-0440},
     journal={J. Geom. Phys.},
      volume={57},
      number={4},
       pages={1125\ndash 1146},
}

\bib{Bi89}{article}{
      author={Bismut, J.-M.},
       title={A local index theorem for non {K}{\"a}hler manifolds},
        date={1989},
     journal={Math. Ann.},
      volume={284},
      number={4},
       pages={681\ndash 699},
}

\bib{Bryant2003}{inproceedings}{
      author={Bryant, R.},
       title={{Some remarks on {$G_2$}-structures}},
        date={2006},
   booktitle={{Proceedings of G\"okova Geometry-Topology Conference 2005}},
   publisher={International Press},
}

\bib{Carrion1998}{article}{
      author={Carri{\'o}n, R.~R.},
       title={A generalization of the notion of instanton},
        date={1998},
     journal={Diff. Geom. Appl.},
      volume={8},
       pages={1\ndash 20},
}

\bib{ChSa02}{inproceedings}{
      author={Chiossi, S.},
      author={Salamon, S.},
       title={{The intrinsic torsion of $SU(3)$ and $G_2$ structures}},
        date={2002},
   booktitle={{Differential Geometry, Valencia 2001}},
   publisher={World Scientific},
       pages={115\ndash 133},
  url={https://www.worldscientific.com/doi/abs/10.1142/9789812777751_0010},
}

\bib{CdBM}{article}{
      author={Clarke, A.},
      author={del Barco, V.},
      author={Moreno, A.~J.},
       title={{${\rm G}_2$}-instantons on 2-step nilpotent {L}ie groups},
        date={2025},
        ISSN={1095-0761,1095-0753},
     journal={Adv. Theor. Math. Phys.},
      volume={29},
      number={3},
       pages={751\ndash 813},
         url={https://doi.org/10.4310/atmp.250724221756},
}

\bib{Clarke2022}{article}{
      author={Clarke, A.},
      author={Garcia-Fernandez, M.},
      author={Tipler, C.},
       title={{ T-Dual solutions and infinitesimal moduli of the
  G$_2$-Strominger system}},
        date={2022},
     journal={Adv. Theor. Math. Phys},
      volume={26},
      number={6},
       pages={1669\ndash 1704},
}

\bib{Lotay2024}{article}{
      author={da~Silva, A.},
      author={Garcia-Fernandez, M.},
      author={Lotay, J.},
      author={S{\'a}~Earp, H.},
       title={{Coupled G${_2}$-instantons}},
        date={2025},
     journal={Int. J. Math.},
       pages={2542002},
}

\bib{hera2026}{incollection}{
      author={de~la Hera, A.},
      author={Galdeano, M.},
      author={Garcia-Fernandez, M.},
       title={{G}$_2$-structures with torsion and the deformed
  {S}hatashvili--{V}afa vertex algebra},
        date={2026},
   booktitle={{2024 MATRIX Annals, Part II}},
   publisher={Springer},
       pages={341\ndash 355},
}

\bib{dlOG21}{article}{
      author={de~la Ossa, X.},
      author={Galdeano, M.},
       title={{Families of solutions of the heterotic G$_2$ system}},
        date={2021},
     journal={arXiv:2111.13221},
}

\bib{magdalena2020}{article}{
      author={de~la Ossa, X.},
      author={Larfors, M.},
      author={Magill, M.},
      author={Svanes, E.},
       title={Superpotential of three dimensional $n= 1$ heterotic
  supergravity},
        date={2020},
     journal={JHEP},
      volume={195},
      number={1},
}

\bib{OssaLaSv15}{article}{
      author={de~la Ossa, X.},
      author={Larfors, M.},
      author={Svanes, E.},
       title={Exploring {SU}$(3)$ structure moduli spaces with integrable
  {G$_2$} structures},
        date={2015},
        ISSN={1095-0761},
     journal={Adv. Theor. Math. Phys.},
      volume={19},
      number={4},
       pages={837\ndash 903},
         url={https://doi.org/10.4310/ATMP.2015.v19.n4.a5},
}

\bib{delaOssa2018}{article}{
      author={de~la Ossa, X.},
      author={Larfors, M.},
      author={Svanes, E.},
       title={{The infinitesimal moduli space of heterotic {${\mathrm{G}}_2$}
  systems}},
        date={2018},
        ISSN={1432-0916},
     journal={Comm. Math. Phys.},
      volume={360},
      number={2},
       pages={727\ndash 775},
         url={https://doi.org/10.1007/s00220-017-3013-8},
}

\bib{dBMR}{article}{
      author={{del Barco}, V.},
      author={{Moroianu}, A.},
      author={Raffero, A.},
       title={{Purely coclosed ${G_2}$-structures on 2-step nilpotent {L}ie
  groups}},
        date={2022},
     journal={Rev. Mat. Complut.},
      volume={35},
       pages={323\ndash 359},
}

\bib{Fernandez2011}{article}{
      author={Fern\'{a}ndez, M.},
      author={Ivanov, S.},
      author={Ugarte, L.},
      author={Villacampa, R.},
       title={Compact supersymmetric solutions of the heterotic equations of
  motion in dimensions 7 and 8},
        date={2011},
        ISSN={1095-0761},
     journal={Adv. Theor. Math. Phys.},
      volume={15},
      number={2},
       pages={245\ndash 284},
         url={http://projecteuclid.org/euclid.atmp/1337951924},
}

\bib{finofow}{article}{
      author={Fino, A.},
      author={Fowdar, U.},
       title={{Some remarks on strong {$\mathrm{G}_2$}-structures with
  torsion}},
        date={2026},
        ISSN={0021-7824},
     journal={J. Math. Pures Appl.},
      volume={211},
       pages={103882},
  url={https://www.sciencedirect.com/science/article/pii/S0021782426000280},
}

\bib{fino2026}{article}{
      author={Fino, A.},
      author={Grantcharov, G.},
      author={Medel, J.},
       title={{New Solutions to the G$_2$ Hull-Strominger system via torus
  fibrations over K$3$ orbifolds}},
        date={2026},
     journal={arXiv:2601.20813},
}

\bib{FinoG2T2023}{article}{
      author={{Fino}, A.},
      author={{Mart{\'\i}n-Merch{\'a}n}, L.},
      author={{Raffero}, A.},
       title={{The twisted G$_2$ equation for strong G$_2$-structures with
  torsion}},
        date={2024},
     journal={Pure Appl. Math. Q},
      volume={20},
      number={6},
       pages={2711\ndash 2767},
}

\bib{Friedrich2001}{article}{
      author={Friedrich, T.},
      author={Ivanov, S.},
       title={Parallel spinors and connections with skew-symmetric torsion in
  string theory.},
        date={2002},
        ISSN={1093-6106},
     journal={Asian J. Math.},
      volume={6},
      number={2},
       pages={303\ndash 335},
}

\bib{FriedrichNP}{article}{
      author={Friedrich, T.},
      author={Kath, I.},
      author={Moroianu, A.},
      author={Semmelmann, U.},
       title={On nearly parallel {$G_2$}-structures},
        date={1997},
        ISSN={0393-0440},
     journal={J. Geom. Phys.},
      volume={23},
      number={3},
       pages={259\ndash 286},
  url={https://www.sciencedirect.com/science/article/pii/S0393044097800046},
}

\bib{GS24}{article}{
      author={Galdeano, M.},
      author={Stecker, L.},
       title={{The heterotic ${\mathrm{G}_2}$ system with reducible
  characteristic holonomy}},
        date={2025},
        ISSN={0393-0440},
     journal={J. Geom. Phys.},
      volume={217},
       pages={105635},
  url={https://www.sciencedirect.com/science/article/pii/S0393044025002190},
}

\bib{Galicki1996}{article}{
      author={Galicki, K.},
      author={Salamon, S.},
       title={{Betti numbers of 3-Sasakian manifolds}},
        date={1996},
        ISSN={1572-9168},
     journal={Geom. Dedicata},
      volume={63},
      number={1},
       pages={45\ndash 68},
         url={https://doi.org/10.1007/BF00181185},
}

\bib{garcia2025parabolic}{article}{
      author={Garcia-Fernandez, M.},
      author={Moreno, A.},
      author={Payne, A.},
      author={Streets, J.},
       title={A parabolic flow for the large volume heterotic {G}$_2$ system},
        date={2025},
     journal={arXiv preprint :2512.14317},
}

\bib{Gon98}{book}{
      author={Gong, M-P.},
       title={{Classification of nilpotent Lie algebras of dimension 7, over
  algebraically closed fields and $\mathbb R$}},
   publisher={PhD Thesis, University Warteloo (Canada)},
        date={1998},
}

\bib{Harland2011}{article}{
      author={Harland, D.},
      author={N{\"o}lle, C.},
       title={Instantons and killing spinors},
        date={2012},
     journal={JHEP},
      volume={2012},
      number={3},
       pages={82},
}

\bib{Hitchin2000}{article}{
      author={Hitchin, N.},
       title={{The geometry of three-forms in six and seven dimensions}},
        date={2000},
     journal={J. Differ. Geom.},
      volume={55},
      number={3},
       pages={547\ndash 576},
}

\bib{hull1986}{article}{
      author={Hull, C.~M.},
       title={Compactifications of the heterotic superstring},
        date={1986},
        ISSN={0370-2693},
     journal={Physics Letters B},
      volume={178},
      number={4},
       pages={357\ndash 364},
  url={https://www.sciencedirect.com/science/article/pii/0370269386913936},
}

\bib{Ivanov2005}{article}{
      author={Ivanov, P.},
      author={Ivanov, S.},
       title={{$\mathrm{SU}(3)$}-instantons and {G}$_2$, {Spin}$(7)$-heterotic
  string solitons},
        date={2005},
     journal={Commun. Math. Phys.},
      volume={259},
      number={1},
       pages={79\ndash 102},
}

\bib{Iva10}{article}{
      author={Ivanov, S.},
       title={Heterotic supersymmetry, anomaly cancellation and equations of
  motion},
        date={2010},
        ISSN={0370-2693},
     journal={Physics Letters B},
      volume={685},
      number={2},
       pages={190\ndash 196},
}

\bib{Lotay2022}{article}{
      author={Lotay, J.},
      author={Sá~Earp, H.},
       title={{The heterotic $\rG_2$ system on contact {C}alabi-{Y}au
  $7$-manifolds}},
        date={2023},
     journal={Trans. A.M.S.},
      volume={10},
       pages={907\ndash 943},
}

\bib{MAL}{article}{
      author={Malcev, A.I.},
       title={On a class of homogeneous spaces},
        date={1962},
     journal={Amer. Math. Soc. Trans. Ser.},
      volume={9},
      number={1},
       pages={276\ndash 307},
}

\bib{Mil76}{article}{
      author={{Milnor}, J.},
       title={{Curvatures of left invariant metrics on Lie groups.}},
        date={1976},
     journal={{Adv. Math.}},
      volume={21},
       pages={293\ndash 329},
}

\bib{MRV}{article}{
      author={Moroianu, A.},
      author={Raffero, A.},
      author={Vezzoni, L.},
       title={The heterotic {$\G_2$}-system on $2$-step nilmanifolds endowed
  with principal torus bundles},
        date={2025},
     journal={arXiv2512.16817},
         url={https://arxiv.org/abs/2512.16817},
}

\bib{Salamon1989}{book}{
      author={Salamon, S.},
       title={Riemannian geometry and holonomy groups},
      series={Pitman Research Notes in Mathematics Series},
   publisher={Longman Scientific \& Technical},
     address={Harlow},
        date={1989},
      volume={201},
        ISBN={0-582-01767-X},
}

\bib{SA1}{article}{
      author={Salamon, S.},
       title={Complex structures on nilpotent {L}ie algebras},
        date={2001},
     journal={J. Pure Appl. Algebra},
      volume={157},
      number={2-3},
       pages={311\ndash 333},
}

\bib{Str86}{article}{
      author={Strominger, A.},
       title={Superstrings with torsion},
        date={1986},
     journal={Nuclear Physics B},
      volume={274},
      number={2},
       pages={253\ndash 284},
}

\end{biblist}
\end{bibdiv}
\bibliographystyle{numeric}

\vskip0.7em
\noindent
V. del Barco, Instituto de Matem\'atica, Estat\'istica e Computa\c{c}\~ao Cient\'ifica, Universidade Estadual de Campinas, Rua Sergio Buarque de Holanda, 651, 13083-859, Campinas, SP, Brazil.
Email: \texttt{delbarc@unicamp.br} 
\vskip0.5em
\noindent
U. Fowdar, 
Institute of Mathematics, University of Warsaw, 
Banacha 2, 02-097 Warszawa, Poland.
Email: \texttt{u.fowdar@uw.edu.pl} 
\vskip0.5em
\noindent
A. J. Moreno, Institut Camille Jordan, Université Claude Bernard Lyon 1, 21 Avenue Claude Bernard, 69100 Villeurbanne, France.
Email: \texttt{amoreno@math.univ-lyon1.fr} 
\end{document}